\documentclass[11pt]{amsart}

\usepackage[margin=1in]{geometry}

\usepackage{amsthm}
\usepackage{breqn}
\usepackage{amsmath}
\usepackage{amsfonts}
\usepackage{amssymb}
\usepackage{tikz}
\usepackage{hyperref}
\usepackage{microtype} 
\usepackage{mathtools}
\usepackage{graphicx}
\usepackage{stmaryrd}
\usepackage[capitalise]{cleveref}

\newcommand{\R}{\mathbb{R}} 
\newcommand{\e}{\varepsilon}
\newcommand{\N}{\mathbb N}

\newcommand{\supp}{\operatorname{supp}}
\newcommand{\dist}{\operatorname{dist}}

\newcommand{\ds}{\delta_\star}

\newcommand{\sgn}{\operatorname{sgn}}

\newtheorem{theorem}{Theorem}[section]
\newtheorem*{theorem*}{Theorem}
\newtheorem{lemma}[theorem]{Lemma}
\newtheorem{corollary}[theorem]{Corollary}
\newtheorem{proposition}[theorem]{Proposition}

\newtheorem{claim}{Claim}
\newtheorem*{claim*}{Claim}

\theoremstyle{definition}

\newtheorem{definition}[theorem]{Definition}
\newtheorem{remark}{Remark}
\newtheorem*{remark*}{Remark}

\title{Multiplicity one and strictly stable Allen-Cahn minimal hypersurfaces}
\author{Marco A. M. Guaraco, Fernando C. Marques, Andre Neves}
\address{Andr\'e Neves - The University of Chicago}
\email{aneves@uchicago.edu}
\address{Fernando Cod\'a Marques - Princeton University}
\email{coda@math.princeton.edu}
\address{Marco Guaraco - Imperial College London}
\email{guaraco@imperial.ac.uk}

\begin{document}

\begin{abstract}
We show that strictly stable components of Allen-Cahn minimal hypersurfaces always occur with multiplicity one. We also establish the uniqueness of solutions converging to non-degenerate hypersurfaces with multiplicity one. Our results work in all dimensions and without variational assumptions on the Allen-Cahn solutions.

\end{abstract}

\maketitle
 

\section{Introduction}

Since Modica-Mortola (\cite{ModicaMortola1,ModicaMortola2}) and De Giorgi (\cite{DeGiorgi}), the mathematical community has been aware of the strong connections between minimal hypersurfaces and functions $u\in C^\infty(M)$ satisfying the semilinear elliptic equation
\begin{align}\label{AC}
\e^2\Delta u -W'(u)=0.
\end{align} Here $\e>0$, $(M^n,g)$ is a closed Riemannian manifold, $\Delta$ is the Laplace-Beltrami operator of $M$, $u\in  C^2(M)$ and $W(t)=(1-t^2)^2/4.$ 

 \cref{AC} is known as the (stationary) Allen-Cahn equation (see \cite{AllenCahn}) and its solutions are critical points of the energy \begin{align}\label{energy}
E_\e(u)=\int_M \e\frac{|\nabla u|^2}{2}+\frac{W(u)}{\e}.
\end{align} A special feature of solutions to \cref{AC} is that, under very general assumptions (see \cite{HutchinsonTonegawa, TonegawaWickramasekera, Guaraco}), their nodal set $\{u=0\}$ converges in the Hausdorff distance to a  minimal hypersurface $\Gamma$ as $\e\to 0$. When this happens we say that $\Gamma$ is a \textit{limit interface}. Moreover, the integrand of the energy converges, in the sense of measures, to a positive integer multiple of the area of $\Gamma$. We call this integer, which can be different on each component of $\Gamma$, the \textit{multiplicity of $\Gamma$ as a limit interface}. \textit{Interface foliation} refers to the situation in which the multiplicity is strictly greater than one in some components.   

\

In this work, we study the phenomenon of interface foliation for solutions of \cref{AC}.  Our interest in the subject comes from its connection with the multiplicity problem for min-max constructions of minimal hypersurfaces, which we briefly describe in the next subsection.

\ 

First, we prove the following \text{local} multiplicity one result for strictly stable minimal interfaces:

\begin{theorem}\label{zerolevelset} Let $U\subset M$ be an open set and $\Gamma\subset U$ a closed embedded minimal hypersurface which is the limit interface of a sequence of solutions to \cref{AC} on $U$. If $\Gamma$ is two-sided and strictly stable, then $\Gamma$ has multiplicity one as a limit interface.
\end{theorem}

 Since one-sided interfaces always occur with multiplicity at least two, lifting the sequence of solutions to their oriented double cover we obtain:

\begin{corollary}
There are no one-sided limit interfaces with strictly stable oriented double cover.  
\end{corollary}

We emphasize that, contrary to what is usually assumed, \cref{zerolevelset} does not impose restrictions on the Morse index of the solutions or the dimension of the ambient manifold. 

\cref{zerolevelset} is proven in \cref{strictlystable}. A key idea is the use of a sliding argument to show that the nodal set of $\{u_\e=0\}$ is contained in a tubular neighborhood of $\Gamma$ of height $O(\e)$. This implies that $\e$-blow-ups of $u_\e$ converge to an entire solution in $\R^n$ whose nodal set is in between two parallel hyperplanes. By the work of Farina (\cite{Farina}, see also \cref{entireclass}), this entire solution has to be the canonical one-dimensional model. This characterization then allows us to bound the energy of $u_\e$ by the area of $\Gamma$. We note that the construction of the subsolution barriers for the sliding argument is possible only because we are assuming $\Gamma$ is strictly stable. These subsolutions are analogue to constant mean curvature hypersurfaces foliating a tubular neighborhood of $\Gamma$ and whose mean curvature vector point towards $\Gamma$. This argument is contained in \cref{nbarriers}. 

\ 

We also obtain the following uniqueness result for multiplicity one solutions:

\begin{theorem}\label{B} 
Let $u_k$ be a sequence of solutions to \cref{AC} with $\e=\e_k\to 0$, whose nodal set $\{u_k=0\}$ converges in the Hausdorff distance to an embedded minimal hypersurface $\Gamma$. If $\Gamma$ is non-degenerate and the convergence is with multiplicity one, then for $k$ large enough, $u_k$ is the one-sheet solution adapted to $\Gamma$ constructed in \cite{Pacard}.  
\end{theorem}  

These one-sheet solutions were constructed by Pacard (\cite{Pacard}) via a contraction argument on a space of functions having specific asymptotics. Our strategy to prove \cref{B} is to show that every multiplicity one solution has the same asymptotics. This is done in \cref{uniqueness}. We work in Fermi coordinates over $\Gamma$, rather than with respect of $\{u=0\}$ as in \cite{WangWei,ChodoshMantoulidis}. These coordinates have the advantage of being the ones used in \cite{Pacard}. We then use the estimates of Wang-Wei (\cite{WangWei}), together with the non-degeneracy assumption, to improve the estimates on the error terms and match the asymptotics of Pacard (\cite{Pacard}).

\subsection{Multiplicity and min-max minimal hypersurfaces}

The min-max approach for variational problems, also known as \textit{mountain-pass}, was first devised by Birkhoff (\cite{Birkhoff}) in order to construct closed geodesics in Riemannian two-spheres. Almgren (\cite{Almgren2}) developed the theory of varifolds in part to generalize Birkhoff's ideas to higher dimensions. These efforts led to the work of J. Pitts (\cite{Pitts}), who exploited the regularity theory of Schoen-Simon-Yau (\cite{SchoenSimonYau}) to show that, for $3\leq n\leq 6$, the varifold obtained through a one parameter min-max, has support on an embedded minimal hypersurface. Shortly after, Schoen-Simon (\cite{SchoenSimon}) developed a regularity theory for stable varifolds in dimensions $n\geq 7$. As a consequence, they extended Pitts' result to all dimensions with the caveat that the minimal hypersurface is now embedded only outside of a set of dimension at most $n-8$. 

Recently, the last two authors proposed a program to extend Almgren-Pitts' ideas to higher parameter min-max families, with the goal of solving Yau's conjecture (\cite{Yau}) on existence of infinitely many immersed minimal hypersurfaces on arbitrary closed manifolds. In (\cite{MarquesNevesInfinitely}), they solved the conjecture for manifolds with positive Ricci curvature. By analogy with the min-max construction of the eigenvalues of the Laplacian, these methods led to the definition of a non-linear spectrum for the area functional. Together with Liokumovich (\cite{LiokumovichMarquesNeves}), they showed this non-linear spectrum satisfies a Weyl law. Later, in joint work with Irie (\cite{IrieMarquesNeves}), they obtained the density of minimal hypersurfaces for generic metrics, and in  joint work with Song (\cite{MarquesNevesSong}), they showed the existence of an equidistributed family of minimal hypersurfaces for generic metrics. Both results based on the Weyl Law, solved Yau's conjecture for generic metrics. Finally, the general case of Yau's conjecture was later solved by Song (\cite{Song}). In this elegant work, Song localized the methods from \cite{MarquesNevesInfinitely} to prove the existence of infinitely many minimal hypersurfaces on a domain bounded by stable hypersurfaces.

All this progress, led the last two authors to propose a Morse-theoretic description of the set of minimal hypersurfaces for generic metrics (see \cite{Marques,MarquesNevesMorse,MarquesNevesTop,Neves}). In this context, it was conjectured:

\vspace{9pt}
\noindent\textbf{Morse Index Conjecture.} \textit{For a generic metric g on $M^n$, $3 \leq n \leq 7$, there exists a sequence $\{\Sigma_k\}$ of smooth, embedded, two-sided, closed minimal hypersurfaces such that: $\operatorname{index}(\Sigma_k) = k$ and
$ C^{-1}k^{1/n} \leq \operatorname{area}(\Sigma _k)\leq C k^{1/n}$ for some $C >0$.}
\vspace{9pt}

The proposed program to prove this conjecture was based on three main components: the use of min-max constructions over multiparameter sweepouts to obtain existence results, the characterization of the Morse index of min-max minimal hypersurfaces under the multiplicity one assumption, and a proof of the following multiplicity one conjecture:

\vspace{9pt}
\noindent\textbf{Multiplicity One Conjecture.} \textit{For generic metrics on $M^{n}$, $3 \leq n\leq 7$, any component of a closed, minimal hypersurface obtained by min-max methods is two-sided and has multiplicity one.}
\vspace{9pt}

In a novel work, X. Zhou (\cite{Zhou}) used a regularization of the area functional (developed by him and Zhu in \cite{ZhouZhu}) to prove the Multiplicity One Conjecture. The characterization of the Morse index of min-max minimal hypersurfaces under the multiplicity one assumption was obtained by the last two authors in \cite{MarquesNevesOne}, completing the Morse-theoretic program they proposed for the area functional. Finally, Wang-Zhou (\cite{Zhou22}) have recently showed the existence of non-generic metrics on $S^n$, $3\leq n \leq 7$, for which a two parameter min-max construction is necessarily attained with multiplicity.

 \subsection{Min-max Allen-Cahn minimal hypersurfaces}
In his PhD thesis (\cite{Guaraco}), the first author proposed a different technical framework for the min-max construction of minimal hypersurfaces. The basic idea is to first use min-max methods to construct solutions of the Allen-Cahn equation having good Morse theoretical properties. Once one constructs solutions, the problem of the convergence of the nodal set towards a minimal hypersurface can be studied separately. In \cite{Guaraco}, the regularity of the limit set is derived from the assumption of bounded Morse index of the solutions, building on the stable case previously handled by Tonegawa-Wickramasekera in \cite{TonegawaWickramasekera}. In this way, the first author was able to obtain a new proof of Almgren-Pitts-Schoen-Simon's Theorem. Later, these ideas were developed further by the first author together with Gaspar (\cite{GasparGuaraco}). In \cite{GasparGuaraco}, an Allen-Cahn spectrum is defined which is analogous to the volume spectrum from  \cite{MarquesNevesInfinitely} and \cite{LiokumovichMarquesNeves}. It is also known that the index of the limit hypersurface is bounded by the index of the solutions. This was shown by Hiesmayr (\cite{Hiesmayr}) for two-sided limits, and by Gaspar (\cite{Gaspar}) in the general case. Later, together with Gaspar (\cite{GasparGuaracoWeyl}), the first author showed that after suitable modifications, the Allen-Cahn spectrum can replace the volume spectrum in the density (\cite{IrieMarquesNeves}) and equidistribution (\cite{MarquesNevesSong}) arguments. 

A stronger regularity theory for stable solutions was subsequently developed for dimension $n=3$ by Chodosh-Mantoulidis (\cite{ChodoshMantoulidis}). Their work is based on Ambrosio-Cabr\'e's characterization of entire stable solutions in $\R^3$ (\cite{AmbrosioCabre}), as well as on improvements of recent regularity estimates for the nodal set of solutions, obtained by Wang-Wei (\cite{WangWei}). Using these estimates they showed that, if multiplicity higher than one occurs, then the limit hypersurface admits a positive Jacobi vector field. In addition, in \cite{ChodoshMantoulidis} they showed the Morse index is lower semicontinuous for multiplicity one solutions. This is the Allen-Cahn analogue of the Morse Index Conjecture for $n=3$. When this regularity theory is applied to the min-max constructions of the first author and Gaspar (\cite{GasparGuaraco}), one obtains an Allen-Cahn analogue to the Multiplicity One Conjecture for $n=3$. 

Finally, we note that parallel ideas to \cite{Guaraco} and \cite{GasparGuaraco}, have been developed for dimension two by Mantoulidis (\cite{Mantoulidis}), and in the codimension two setting by Stern (\cite{Stern}) and Pigati-Stern (\cite{PigatiStern}). In \cite{PigatiStern} they prove the existence of codimension two integer rectifiable varifolds obtained as the limit interface of min-max critical points of a complex valued functional. As a result, they obtain a new proof of the existence of stationary integral $(n-2)$-varifolds in an arbitrary closed Riemannian manifold. Although this result was already known to Almgren (\cite{Almgren}), the proof presented in \cite{PigatiStern} is considerably less technically involved, opening possibilities for new developments in the field.

\
\subsection{Interface foliation on Allen-Cahn minimal hypersurfaces}

Existence of solutions with nodal set near certain minimal hypersurfaces, has been proven for the multiplicity one case in \cite{PacardRitore,Pacard} and more recently in \cite{CajuGaspar}.  Interface foliation was studied in \cite{DelPinoKowalczykWeiYang} for non-degenerate separating hypersurfaces $\Sigma$ satisfying the second order condition $|A|^2+\operatorname{Ric}(\nu,\nu)>0$, where $A$ and $\nu$ are the second fundamental form and normal vector of the hypersurface. This inequality implies that the Jacobi operator of $\Sigma$ is unstable. As mentioned in \cite{DelPinoKowalczykWei}, this construction works also when $\Sigma$ is non-separating, but in this case, the multiplicity of the interface must be even. 

In \cite{DelPinoKowalczykWeiYang}, it is shown that for any hypersurface satisfying the condition above and any $k\in\N$, there exists a sequence of $\e\to 0$ and solutions of the Allen-Cahn, whose nodal set is $k$ small graphs accumulating on $\Sigma$.

For convenience of the reader, we include a sketch of a different construction which works for some minimal hypersurfaces of $\mathbb{S}^n$, $\mathbb{R}\mathbb{P}^n$ and a torus. In these highly symmetric geometries, solutions exist for $\e$ sufficiently small, rather than for only a subsequence going to zero.

\subsection*{Example 1.} Let $\mathbb{S}^n=\{\|x\|=1: x\in \R^{n+1}\}$. Given $\tau\in(0,1)$, we  partition $\mathbb{S}^n$ into the sets $D_\tau^+=\mathbb{S}^n\cap \{x_{n+1}\geq \tau \}$, $A_\tau=\mathbb{S}^n\cap \{|x_{n+1}|\leq \tau \}$ and $D_\tau^-=\mathbb{S}^n\cap \{x_{n+1}\leq -\tau \}.$ Let $\Omega$ be any of these domains. We can minimize the energy $E_\e$ on $W^{1,2}_0(\Omega)$. This produces a solution with zero Dirichlet condition at $\partial\Omega$. By comparing with a positive function, we can see that there exists $\e_0>0$, such that for any $\e\in(0,\e_0)$, there are non-zero minimizers in $A_\tau$ iff $\tau$ is far from zero, and in $D_\tau^\pm$ iff $\tau$ is far from one. These minimizers must have a sign, otherwise, due to the symmetry of $W$, $|u|$ would also minimize, contradicting the maximum principle since $|u|\geq 0$ and $0$ is also a solution. By \cite{BrezisOswald} positive solutions are unique, so the positive minimizer must be rotationally symmetric. By continuity, for each $\e\in(0,\e_0)$, there is a $\tau_0$ such that the outer derivative of these solutions coincide on $\partial (D_\tau^-\cup D_\tau^+)=\partial A_\tau$. If we choose the positive minimizer for $A_{\tau_0}$ and the negative for $D_{\tau_0}^\pm$ we obtain a solution with nodal set equal to two parallels equidistant from the equator. This implies a lower bound on the energy (see \cite{GasparGuaraco}). For some subsequence of $\e\to 0$, the parallels either converge to the equator, or accumulate on two distant parallels. However, by the results of Hutchinson-Tonegawa (\cite{HutchinsonTonegawa}) the energy must accumulate both near the zero level set and near a stationary varifold. Since two parallels are not stationary, convergence towards the equator is the only possibility.

\subsection*{Example 2.} The solutions constructed above are also even with respect to the antipodal map $x\to-x$. In particular, they project well to $\R\mathbb{P}^n$ giving solutions accumulating with multiplicity one on a copy of $\R\mathbb{P}^{n-1}$. We notice that Example 1 can be reproduced on a bumpy, rotationally and antipodally symmetric ellipsoid. In this case, projecting into the quotient gives a non-orientable strictly stable Allen-Cahn minimal hypersurface with unstable double cover.

\subsection*{Example 3.} Let $\mathbb{T}^n=\mathbb{T}^{n-1}\times S^1$ be a rotationally symmetric torus having only two  minimal leaves: $T_1=\mathbb{T}^{n-1}\times \{a\}$, which is unstable, and $T_0=\mathbb{T}^{n-1}\times \{b\}$, which is strictly stable. \cref{zerolevelset} shows that $T_0$ is not an Allen-Cahn minimal hypersurface. Under enough symmetries, the construction from Example 1 can be adapted to this case in order to show that, for every $\e$ small enough, there is a solution with nodal set two normal graphs over $T_1$. Therefore, $T_1$ can be obtained as an Allen-Cahn minimal hypersurface with multiplicity two. If $T_1$ is  non-degenerate case this last statement also follows from \cite{DelPinoKowalczykWeiYang}. If the union $T_0\cup T_1$ is non-degenerate then it can also be obtained as a multiplicity one Allen-Cahn minimal hypersurface by \cite{PacardRitore}. We do not know whether $T_0\cup T_1$ could be obtained with $T_0$ having multiplicity 1 and $T_0$ having odd multiplicity strictly greater than 1.

\section{Organization}

 \cref{notation} summarizes most of the notation used along the paper. \cref{canonicalandcutoff} presents the fundamental estimates for the canonical solutions as well as its cutoffs.  \cref{Fermicoordinates}, discusses Fermi coordinates.  \cref{injectresults} collects the fundamental injectivity results for the linearized equation.  \cref{subandsuper} describes elementary properties of subsolutions to the Allen-Cahn equation. \cref{ellipticestimates} collects standard elliptic estimates suited to our notation.  \cref{entireclass} describes known characterizations of entire solutions.   \cref{strictlystable} contains the proof of the multiplicity one in \cref{zerolevelset}.  \cref{curvature} summarizes curvatures estimates for the multiplicity one case following \cite{Wang} and \cite{WangWei}. Finally, \cref{uniqueness} finishes the proof of \cref{zerolevelset} and \cref{B}.

\subsection*{Acknowledgements} We would like to thank the reviewers for their thoughtful questions and inputs towards
improving our exposition.

\section{Notation}\label{notation}

We use the big $O$ and little $o$ notation with respect to the variable $\e$. Let $\Omega$ be an open region of a Riemannian manifold $M$ and $\e_k \to 0$, $k\in\N$. From now on, we omit the reference to the index of the sequence and write $\e=\e_k$. In addition, when working with a sequence of functions $f_{\e_k}$, we also omit the reference to $\e$, setting $f=f_{\e_k}$.

\begin{definition}\label{Oo}
Given a sequence of functions $f :\Omega\to \R$, we say that 
 \begin{itemize}
\item $f=O(1)$, \ \ if $\limsup_{\e\to0} \|f\|_{L^\infty(\Omega)}<\infty$,
\item $f=o(1)$, \ \ if $\lim_{\e\to0} \|f\|_{L^\infty(\Omega)}=0$, 
\item $f=O(g)$, \ \ if $|f|=O(1)\times |g|$, and  \item $f=o(g)$, \ \ if $|f|=o(1)\times|g|$.
\end{itemize}
Additionally, we make the following convention: a function $f \in C^\infty(\Omega)$ is said to be of class $o(\e^\N)$ in $\Omega$, if all of its derivatives and integrals on the set $\Omega$, decay faster than polynomials on $\e$, i.e. 
\begin{itemize}
\item $f=o(\e^\N)$, \ \ if $\|\nabla^k f\|_{L^{p}(\Omega)}=o(\e^m),$ for any fixed values of $k,m \in \{0,1,2,\dots\}$ and $p \in [1,\infty].$
\end{itemize}
\end{definition}

\begin{definition}\label{cke}
Given an open region on Riemannian manifold $(\Omega,g)$, $k\in \N$, $\alpha\in[0,1)$ and $\e>0$, we denote by $C^{k,\alpha}_\e (\Omega)$ the H\"older space $C^{k,\alpha}(\Omega)$ endowed with the (rescaled) norm  of $C^{k,\alpha}(\Omega,\e^{-2}g)$. In other words, given $f\in C^{k,\alpha}(\Omega)$, we define $$\| f \|_{C^{k,\alpha}_\e(\Omega)}= |f|_{C^0(\Omega)} +  \e |\nabla f|_{C^0(\Omega)} + \cdots +\e^k |\nabla^k f|_{C^0(\Omega)}+ \e^{k+\alpha} [|\nabla^k f|]_{\alpha},$$ where $[f]_{\alpha}:=\sup_{x,y\in \Omega} \frac{| f(x)-f(y)|}{\operatorname{dist}_\Omega(x,y)^\alpha}.$
\end{definition}

We also use the following symbols:

\medskip

\begin{tabular}{lll}
$o(\e^\N)$ && (see \cref{Oo}).\\
$C^{k,\alpha}_\e$ && (see  \cref{cke}).\\
$W$ && denotes the canonical potential $W(x)=\frac{(1-x^2)^2}{4}$.\\
$Q(u)$ && is the Allen-Cahn operator $\e^2 \Delta u-W'(u)$.\\
$\psi$ && denotes the one dimensional solution (see \cref{onedim}) \\
$\sigma_0$ && is the energy constant $\int_{-1} ^{1}\sqrt{2W(s)}ds$.\\  
$\sigma_1$ && is the constant $\int_\R \psi'$.\\
$\sigma_2$ && is the constant $\int_\R (\psi')^2$.\\  
$B_R(p)$ && denotes the ball of radius $R$ centered at $p\in M$.\\
$\sgn$ && denotes the sign function $\sgn:\R \to \{-1,0,1\}$.\\
$\Gamma(f)$ && represents the normal graph of $f$ over a two-sided hypersurface $\Gamma$. \\
$\{|t|<h\}$ && denotes the tubular neighborhood of height $h$\\ && where $t$ is the height of the Fermi coordinates over a hypersurface. \\
$\omega$ && denotes a cutoff of the canonical solution\\
$\ell_0$ && denotes the 1-D linearized operator $\e^2 \partial_t^2 - W''(\psi(t/\e))$. \\
$\ell$ && denotes the 1-D approximate linearized operator $\e^2 \partial_t^2 - W''(\omega)$. \\
$L_0$ && denotes $\e^2 \Delta - W''(\psi(t/\e))$. \\
$L$ && denotes the approximated linearized operator $\e^2 \Delta^2 - W''(\omega)$. \\
 \end{tabular}

\section{Properties of the canonical solution}\label{canonicalandcutoff}

We denote by $\psi:\R\to(-1,1)$ the one dimensional canonical solution, i.e. the unique entire solution to 
\begin{align*}
\begin{cases}
\psi'' - W'(\psi)=0 \\
\psi(0)=0\\
\psi'>0,
\end{cases}
\end{align*} where $W(u)=(1-u^2)^2/4.$ It is well known that $\psi(t)=\tanh(t/\sqrt{2})$. By direct differentiation one verifies that $\psi-\sgn$ decays exponentially at infinity together with all of its derivatives. More precisely, there are constants $\sigma>0$ and $c_k$, for $k \in \N$, such that 
$|(\psi-\sgn)(t)| \leq c_0 e^{-\sigma |t|}$ and
$|(\partial^k \psi)(t)| \leq c_k e^{-\sigma |t|}.$

Often, we will work with the rescaled functions $\psi(t/\e)$ that satisfies 
\begin{align}\label{onedim}
\e^2 (\psi(t/\e))''-W'(\psi(t/\e))&=0
\end{align} and the estimates
\begin{align}\label{onedimest}\begin{split}
|\psi(t/\e)-\sgn(t)| &\leq c_0 e^{-\sigma |t|/\e}\\
|(\partial^k \psi)(t/\e)|  &\leq c_k e^{-\sigma |t|/\e}.
\end{split}\end{align}

\subsection{A cutoff of the canonical solution} In later sections, we will use $\psi$ as a model to construct approximate solutions adapted to small tubular neighborhoods of smooth hypersurfaces. For this purpose, we cutoff $\psi(t/\e)$ on small regions of order greater than $O(\e)$.

For every $\e>0$, we define the cutoff $\omega:\R\to\R$ as 
\begin{align}\label{cutoffdef}\begin{split}
\omega(t)&=\psi(t/\e)\chi+(1-\chi(t))\sgn(t)\\
&=\psi(t/\e) + (\psi(t/\e)-\sgn(t))(\chi(t)-1)
\end{split}\end{align}
where $\chi(t)=\rho(2 \e^{-\delta} |t| -1)$, with $\rho:\R \to \R$ a smooth bump function satisfying $\rho \geq0$, $\rho(t)=1$ for $t\leq0$ and $\rho(t)=0$ for $t\geq 1$, and $\delta\in(0,1)$. 

It follows that
$$\chi(t)=
\begin{cases}
1 & \text{ for } |t| \leq \e^{\delta}/2\\
0 & \text{ for } |t| \geq \e^{\delta}
\end{cases}$$ and all the derivatives of $\chi$ are supported on $\e^{\delta}/2\leq |t| \leq \e^{\delta}$ and bounded by rational functions of $\e$. Namely, for all $k\in\N$, we can assume that \begin{align}\label{chiderivatives}|\partial^k \chi|\leq c_k \e^{-k\delta} \text{ on  }  [\e^{\delta}/2,\e^{\delta}],\end{align} for some $c_k\in \R$, and $\partial^k \chi=0$, otherwise.

The following are simple consequences of direct differentiation of \cref{cutoffdef}, together with the estimates from \cref{onedimest} and \cref{chiderivatives}.

\begin{lemma}\label{cutoffestimates}

\

\begin{enumerate} 
\item $(\psi(t/\e)-\sgn(t))(\chi(t)-1)=o(\e^\N)$. \\
\item $\e^k \partial^k \omega(t)=(\partial^k \psi)(t/\e)+o(\e^\N)$, for $k\geq 0$. \\
\item $|\omega(t)-\sgn(t)|+|\e^k \partial^k\omega(t)|=O(e^{-|t|/\e})$.\\
\item $\bigg| \int_\R f ( \omega -\sgn) \bigg| +\e^k \bigg|  \int_\R f \partial^k \omega \bigg| =O(\e) \|f\|_{L^\infty(\R)} $. \\
\item  If $\|f\|_{L^{\infty}([-R\e^{\delta},R\e^{\delta}])}=o(1)$, for some $R>0$ and $\delta\in(0,1)$, then (4) improves to 
$\bigg| \int_\R f ( \omega -\sgn) \bigg| +\e^k \bigg|  \int_\R f \partial^k \omega \bigg| =  o(\e).$\\
\item If $f=|t|^p$, then (4) improves to $\int_\R |t|^p|\omega -\sgn| +\e^k \int_\R |t|^p |\partial^k \omega| =O(\e^{1+p})$,  for all $p,k\in \N$.\\
\item $ \int_\R \partial^k\omega \partial^m\omega =\begin{cases}c_{k,m} \cdot \e^{1-k-m} + o(\e^\N)& \text{if $k$ and $m$ have the same parity}\\ 0 & \text{otherwise}\end{cases}$ 
where $c_{k,m}\in \R$ are universal constants and for $k=m=1$, $\sigma_2=c_{1,1}>0$.\\
\end{enumerate}
\end{lemma}

\subsection{The linearized equation} We denote by $\ell_0$ the linearized Allen-Cahn operator, i.e. \begin{align*}
\ell_0(\varphi) =\e^2( \varphi )''- W''(\psi(t/\e))\varphi. 
\end{align*} 
Direct differentiation of \cref{onedim}, shows that $\psi'(t/\e)$ is a positive function on the kernel of $\ell_0$, i.e. $\ell_0(\psi'(t/\e))=0$. It is a well known fact, that $\operatorname{ker}\ell_0$ is simple and there is a spectral gap for functions in $(\operatorname{ker}\ell_0)^\perp$. The following result is proven in \cite{Pacard} (see equation (3.15)  after Lemma 3.6 in \cite{Pacard}).

\begin{lemma}\label{spectralgap} The function $\psi'(t/\e)$ generates $\operatorname{ker}\ell_0$. Moreover, there exists $\gamma>0$ such that $$ \gamma \int_\R \phi^2 \leq  -\int_\R \phi \ell_0 (\phi),$$ $\forall \phi\in(\operatorname{ker}\ell_0)^\perp$, i.e. $\int_\R \phi(t) \psi'(t/\e)=0$. \end{lemma}

The linearized equation also approximates well by means of the cut-off. \begin{definition}
We define the approximate one dimensional operator as $\ell(\varphi)=\e^2 \varphi''-W''(\omega)\varphi$.
\end{definition} Notice that $\omega(t)^m= \psi(t/\e)^m+ o(\e^\N).$
Since $W''(t)=3t^2-1$ is a polynomial, we obtain $
|\ell(\varphi)-\ell_0 (\varphi)|=|W''(\psi(\cdot/\e))-W''(\omega)||\varphi| =o(\e^\N)\|\varphi\|_{L^{\infty}}.$

\begin{lemma}\label{linearizedapprox}
Let $\varphi \in C^2(\R)$. Then, $$\frac{\gamma}{2} \int_{\R}\varphi^2 \leq -\int_\R \varphi \ell (\varphi) + O (\e) \bigg[\int_\R \varphi \omega' \bigg]^2+ o(\e^\N)\|\varphi\|^2_{L^\infty(\R)}.$$ \end{lemma}

\begin{proof}

Define $\varphi^\perp$ by the formula
\begin{align*}
\varphi(t)= \bigg(\frac{\int_\R \varphi(s) \psi'(s/\e) ds}{\int_\R \psi'(s/\e)^2 ds}\bigg)\psi'(t/\e) + \varphi^\perp(t)
\end{align*} and notice that $\varphi^\perp \in (\ker \ell_0)^\perp$.

Since $\ell_0(\psi'(t/\e))=0$ and $\ell_0$ is self-adjoint, 
 \begin{align*}
\gamma \int_\R (\varphi^\perp)^2 &\leq -\int_\R \varphi^\perp \ell_0 (\varphi^\perp)\\
&= -\int_\R \varphi \ell_0 (\varphi)\\
&= -\int_\R \varphi \ell (\varphi) + o(\e^\N)\|\varphi\|^2_{L^\infty(\R)}.
 \end{align*}
 
Finally, since $\int_\R \psi'(t/\e)^2=O(\e)$ we have,
\begin{align*}
\int_{\R} |\varphi-\varphi^\perp|^2 &=O(\e^{-1})\bigg(\int_\R \varphi(t) \psi'(t/\e) dt \bigg)^2\\
&=O(\e)\bigg(\int_\R \varphi [\omega' + o(\e^\N)] dt \bigg)^2.
 \end{align*}
The result follows from combining both estimates and the properties of $o(\e^\N)$ functions.
\end{proof}


\section{Fermi Coordinates}\label{Fermicoordinates}
Let $\Gamma\subset M$ be a \textit{two-sided} embedded hypersurface. We will often work using \textit{Fermi coordinates} over $\Gamma$, i.e. given a choice of normal frame $\partial$ on $\Gamma$, the coordinates are given by the diffeomorphism $\mathcal{F}:\Gamma \times (-\tau,\tau) \to M$, $$\mathcal{F}(x,t)=\operatorname{Exp}(x,t\partial(x)),$$ for some small $\tau>0$ fixed. \\

For the convenience of the reader, we summarize our notation and several well known facts about Fermi coordinates, in the list below.

\

\begin{itemize}
\item $\{|t|<s\}=\mathcal{F}(\Gamma\times(-s,s))$ denotes the tubular neighborhood of height $s \in (0,\tau)$.\\
\item From now on, given a function $G:\{|t|<s\}\to \R$ we will abuse notation and also denote $G\circ \mathcal{F}$ as $G$.\\
\item $G'(x,t)=(\partial_t G)(x,t)$ denotes the normal derivate of $G\in C^\infty(\{|t|<\tau\})$ at the point $\mathcal{F}(x,t)$.\\
\item $\Gamma(f)=\{\mathcal{F}(x,f(x)):x\in \Gamma\}$ denotes the normal graph of $f:\Gamma\to(-\tau,\tau)$.\\
\item $\nabla_t$ is the gradient operator of $\Gamma(t)$ with respect to the metric inherited from $M$. \\ 
\item $\Delta_t$ is the Laplace-Beltrami operator of $\Gamma(t)$ with respect to the metric inherited from $M$.\\
\item $H_t(x)=H(x,t)$ is the mean curvature of $\Gamma(t)$ at the point $\mathcal{F}(x,t)$ (in the direction of $\partial_t$). We abbreviate $H=H(\cdot,0)$, $H'_0=H'(\cdot,0)$ and $H''_0=H''(\cdot,0)$. \\
\item $J=\Delta_0 + H'_0$ is the Jacobi operator of the hypersurface $\Gamma$, i.e. the second derivative of the area element in the direction of $\partial_t$. \\
\item The ambient Laplace-Beltrami operator decomposes through the well-known formula $$\Delta_g=\Delta_t +\partial_t^2 - H_t\partial_t.$$
\item Given a coordinate system $\partial_i$ on $\Gamma$ we have $(\Delta_t v)(x) = a_{ij}(x,t)(\partial_{ij}v)(x,t)+b_{i}(x,t)(\partial_{i}v)(x,t)$ and $(\nabla_t v)(x) = c_{i}(x,t)(\partial_{i}v)(x,t)$, for $a_{ij},b_i$ and $c_i$, smooth functions on $\Gamma(-\tau,\tau)$.\\
\end{itemize}

We record now the following estimates 
\begin{lemma}\label{nnorms}
Let $G \in C^{0,\alpha}(\{|t|<\tau\})$, then we have the following:
\begin{enumerate}
\item $\|G(\cdot,t)\|_{C^{0,\alpha}_\e(\Gamma)}=O(\|G\|_{C_\e^{0,\alpha}(\{|t|<\tau\})}),$ for all $|t|\leq \tau$.\\
\item $\|G(x,t+\xi(x))\|_{C^{0,\alpha}_\e(\{|t|<\tau/2\} )}=O(\|G\|_{C_\e^{0,\alpha}(\{|t|<\tau\})}),$ for any $\|\xi\|_{C^{1}_\e(\Gamma)}=O(\e)$.\\
\item $\big\|\int_\R G(\cdot,t)g(t) dt\big\|_{C_\e^{0,\alpha}(\Gamma)}=O\big( \int_\R \|G(\cdot,t)\|_{C_\e^{0,\alpha}(\Gamma)}  |g(t)|dt\big),$ for any $g$ with $\supp g \subset [-\tau,\tau].$\\
\item $\big\|\frac{1}{t}(\e^2\Delta_t v-\e^2\Delta_0 v)\big \|_{C^{0,\alpha}_\e(\Omega)}=O(\|v\|_{C^{2,\alpha}_\e(\Omega)})$, for any $v \in C^{2,\alpha}(\{|t|<\tau\})$ and $\Omega\subset \{|t|<\tau\}$.\\
\item $\big\|\frac{1}{t}(\e^2|\nabla_t v|^2-\e^2|\nabla_0 v|^2)\big\|_{C^{0,\alpha}_\e(\Omega)}=O(\|v\|^2_{C^{1,\alpha}_\e(\Omega)})$, for any $v \in C^{2,\alpha}(\{|t|<\tau\})$ and $\Omega\subset \{|t|<\tau\}$.
\end{enumerate}
\end{lemma}

\begin{proof} (1), (3) and the formula $[G(x,t+\xi(x))]_{0,\alpha}=O([G]_{0,\alpha}(1+|\nabla \xi|^\alpha))$ (from which (2) follows) can be derived directly from the definitions of the H\"older norms. For (4), notice that in coordinates we have expressions of the form
 \begin{align*}
\e^2(\Delta_t-\Delta_0) v(x,t)&=\e^2 [A_{ij}(x,t) (\partial_{ij}v)(x,t)+B_{i}(x,t) (\partial_{i}v)(x,t)] \times t
\end{align*}
and
\begin{align*}
\e^2(|\nabla_{t}v(\cdot, t)|^2-|\nabla_{0}v(\cdot, t)|^2)&=\e^2 C_{ij}(x,t)\times \partial_i v (x,t) \times \partial_j v(x,t)\times t,
 \end{align*} where $A_{ij}(x,t)=\int_0^1 a'_{ij}(x,ts)ds$, $B_{ij}(x,t)=\int_0^1 b'_i(x,ts)ds$ and $C_{ij}(x,t)=\int_0^1 c'_i(x,ts)ds$ are depending only on the metric and $\Gamma$.
\end{proof}

\subsection{CMC near non-degenerate minimal hypersurfaces}\label{non-degenerate}
When the Jacobi operator of $\Gamma$ is invertible, a standard application of the Inverse Function Theorem gives the existence of positive constants $\tau=\tau(M,\Gamma)$ and $C=C(M,\Gamma,\tau)$, such that for all $H \in (-\tau, \tau)$:\\
\begin{itemize}
\item There is a unique hypersurface $\Gamma_H$, which is a normal graph over $\Gamma$ and has constant mean curvature equal to $H$.\\
\item The graph $\Gamma_H$ varies smoothly with respect to $H$.\\
\item The distance function $\operatorname{dist}(\cdot, \Gamma_H)$ is smooth on $\{|t|<\tau\}$.\\
\item The map $\mathcal{F}_0$ given in Fermi coordinates with respect to $\Gamma_H$ is a diffeomorphism to $N_H(\tau)=\mathcal{F}_0(\Gamma_H\times (-\tau,\tau)).$\\ 
\item $C^{-1}\|G\circ \mathcal{F}_0\|_{C^k(\Gamma_H\times (-\tau,\tau))}\leq \|G\|_{C^k(N_0(\tau))}\leq C \|G\circ \mathcal{F}_0\|_{C^k(\Gamma_H\times (-\tau,\tau))},$ for $k=1,2,3$.
\end{itemize}


\section{Injectivity results}\label{injectresults}

\subsection{The case of a cylinder}\label{cylinder}

Let $(\Gamma,h)$ be a closed $(n-1)$-dimensional Riemannian manifold. The Laplace-Beltrami operator of the cylinder $\Gamma\times\R=\{(x,t): x\in \Gamma, t\in \R\}$, endowed with the product metric, decomposes as $\Delta_{\Gamma\times\R}=\Delta_0 + \partial_t^2$. 

In this context, the function $\psi(t/\e)$, satisfies $$\e^2 \Delta_{\Gamma\times \R} (\psi(t/\e))-W'(\psi(t/\e))=0$$ and its linearized operator at $\psi(t/\e)$ is given by $$L_0=\e^2 \Delta_{\Gamma\times \R}  - W''(\psi(t/\e)).$$ By differentiating the equation for $\psi(t/\e)$ with respect to the normal direction, we see that $L_0(\psi'(t/\e))=0$, i.e. $\psi'(t/\e) \in \operatorname{Ker}(L_0)$. In fact, Lemma 3.7 from \cite{Pacard}, implies $\operatorname{Ker}(L_0)=\operatorname{span}\langle \psi'(t/\e) \rangle$. 

We say that a function $f\in C^{2,\alpha}(\Gamma\times\R)$ on the cylinder is orthogonal to the kernel of $L_0$ (or orthogonal to $\psi'(t/\e)$), if
\begin{align} \label{Ortho0}
\int_\R f(x,t)\psi'(t/\e) dt = 0, \ \ \text{$\forall x\in \Gamma$.}
\end{align} Direct computation shows that functions orthogonal to the kernel form an invariant subspace of $L_0$. The invertibility properties of $L_0$ on this subspace  are summarized in the following statement, which combines Propositions 3.1 and 3.2 from \cite{Pacard}.

\begin{proposition}\label{injectivitycylinder} Let $(\Gamma,h)$ be a closed $(n-1)$-dimensional Riemannian manifold. There are positive constants $C$ and $\e_0$, such that for all $\e\in(0,\e_0)$
\begin{enumerate} \item If $v\in C^{2,\alpha}_\e(\Gamma\times\R)$ satisfies \cref{Ortho0} then, $$\| v \|_{C^{2,\alpha}_\e(\Gamma\times\R)}\leq C \| L_0 v \|_{C^{0,\alpha}_\e(\Gamma\times\R)}.$$ 
\item If $f\in C^{0,\alpha}_\e(\Gamma\times \R)$ satisfies \cref{Ortho0}, then there exists a unique $v\in C^{2,\alpha}_\e(\Gamma\times\R)$ satisfying \cref{Ortho0} such that  $L_0 v =f$.\end{enumerate} 
\end{proposition}

The following result, which is proved in Proposition 3.3 of \cite{Pacard}, regards the coercive operator that approximates $L_0=\e^2 \Delta -W''(\psi(t/\e))$ at infinity. In fact, notice that $W''(\psi(t))\to 2$ as $t\to\infty$.

\begin{proposition}\label{coercive}
For any $\e>0$, the operator $L_\infty=\e^2\Delta_g - 2$ is an isomorphism $L_\infty : C^{2,\alpha}(M) \to C^{0,\alpha}(M)$ with inverse bounded with respect to the rescaled norms, i.e. $$\| f \|_{C^{2,\alpha}_\e(M)}=O(\| L_\infty f \|_{C^{0,\alpha}_\e(M)}). $$ 
\end{proposition}


\section{Elementary properties of solutions and subsolutions}\label{subandsuper}

The existence of positive solutions for the problem with Dirichlet boundary data
\begin{equation}\label{DAC}\begin{cases}
-\Delta u + W'(u)=0 & \text{ on } \Omega \\
u=0 &\text{ on } \partial \Omega,
\end{cases}\end{equation}
depends on the region being \textit{large enough}, in the sense that the first eigenvalue for the Laplacian has to be small. The  following of proposition follows from the results in \cite{Barlow2000} (see also Proposition 2.4 from \cite{GasparGuaraco}).

\begin{proposition}\label{existence}
Let $\Omega\subset M$ be a bounded open region with Lipschitz boundary and first eigenvalue $\lambda_1=\lambda_1(\Omega)$. There exists a unique positive solution of \cref{DAC} if and only if $\lambda_1<-W''(0)$. Moreover, $u$ and $-u$ are the unique global minima of the energy $E_1$ in $H^1_0(\Omega)$.
\end{proposition}

\begin{proof} We summarize the ideas of the proof. For one direction, note that if the constant function $0$ is a strictly unstable critical point for $E_1$, then there must be a non-zero minimizer. The instability is precisely, $E''(0)(\phi,\phi)<0$ which implies $\lambda_1\leq \int |\nabla \phi|^2/\int \phi^2<-W''(0)$. For the other direction, assume that $u>0$ is a solution. Then $\lambda_1 \int u^2 \leq \int |\nabla u|^2=\int u^2(-W'(u)/u)$. Since $-W'(u)/u$ is monotone decreasing in $[0,1]$ and $W'(0)=0$, it follows that $\lambda_1\leq -W''(0)$. The uniqueness of the positive solution follows from \cite{Barlow2000}.
\end{proof}

The following is just Serrin's maximum principle (see \cite{FangHuaLi}).

\begin{proposition}\label{maxprinciple}
Let $u$ be a supersolution of \cref{AC} in $\Omega$ and $v$ be a subsolution of \cref{AC} in $\Omega$, i.e. $\Delta u - W'(u) \leq 0,$ and $\Delta v - W'(v) \geq 0.$ If $u\geq v$ in $\Omega$, then either $u=v$ or $u>v$ in $\Omega$.
\end{proposition}

In our context, this proposition can be applied together with the following standard lemma.

\begin{lemma}\label{theta}
Let $u$ be a positive subsolution of \cref{AC} on $\Omega$. Then, for all $\theta\in(0,1]$, the function $\theta u$ is also subsolution. Similarly, if $u$ is a positive supersolution, then for all $\theta\in [1,+\infty)$, the function $\theta u$ is also a supersolution.
\end{lemma}

\begin{proof}
Assume that $u$ is a positive subsolution. Let $\theta\in(0,1)$ and $x$ such that $u(x)>0$. Then 
$$
\Delta (\theta u(x))\geq \theta W'(u(x))=\theta u(x) \frac{W'(u(x))}{u(x)}\geq W'(\theta u(x)),$$
where the inequality comes from the monotonicity of $W'(t)/t=t^2-1$. The proofs for the supersolution case  is analogous.
\end{proof}

Using this fact we obtain the following useful result.

\begin{corollary}\label{below}
Let $\Omega \subset M$ be a bounded open region with smooth boundary. Let $u$ and $v$ be solutions to \cref{AC} in $\Omega$. If $u$ is continuous and positive on $\overline{\Omega}$ and $v$ has Dirichlet boundary data equal to $0$, then $u>v$.  
\end{corollary}

\begin{proof}
Since $\overline{\Omega}$ is compact there is $\alpha>0$ such that $u\geq \alpha$ everywhere. In particular, for small values of $\theta>0$ one must have $\theta v < u$ on $\Omega$. This inequality also holds for $\theta=1$. If not, by continuously making $\theta \to 1$ from below, one would find a first point of contact of the graphs of $\theta v$ and $u$. Since, by  \cref{theta}, $\theta v$ is a subsolution for $\theta\in[0,1]$, this would contradict the maximum principle, i.e. \cref{maxprinciple}.

\end{proof}


\section{Standard elliptic estimates}\label{ellipticestimates}
In this section, we summarize the elliptic estimates we will use in the proofs contained in the next sections. In what follows, $\Omega \subset M$ denotes a Lipschitz open region of a fixed closed Riemannian manifold and $L_\e v$, a linear elliptic operator  of the form \begin{align}\label{operatorL}L_\e v =\e^2 \Delta_g v - c(x) v.\end{align}

\begin{theorem}[Estimates for weak solutions]\label{NashSchauder} There exists $\e_0=\e_0(M)>0$, such that for all $\e\in(0,\e_0)$ the following holds. 

Assume $f \in L^{\infty}(\Omega)$ and $\|c\|_{L^{\infty}(\Omega)}\leq K$, for some $K>0$. If  $v \in W^{1,2}(\Omega)$ is a weak solution of $L_\e v =f,$ we have for any $\Omega'\subset\subset \Omega$ the estimate \begin{align*} \|v \|_{C^{0,\alpha}_\e (\Omega')} & \leq C ( \e^{-n/2} \|v \|_{L^2(\Omega)}+\|f\|_{L^\infty(\Omega)}) \\ & \leq C ( \e^{-n/2}\|v \|_{L^\infty(\Omega)}+\|f\|_{L^\infty(\Omega)}), \end{align*} where $C=C(M, K, \e^{-1}d',\operatorname{Vol}(\Omega))>0$, $\alpha=\alpha(M,\e^{-1}d')$ and $d'= \operatorname{dist}(\Omega',\partial\Omega)$.
\end{theorem}

\begin{theorem}[Schauder estimates]\label{Schauder} Given $v \in C^{2,\alpha}(\Omega)$ and $\Omega'\subset\subset \Omega$ we have the estimate \begin{align*} \|v \|_{C^{2,\alpha}_\e(\Omega')} \leq C ( |v |_{C^{0}(\Omega)}+\|L_\e v\|_{C^{0,\alpha}_\e(\Omega)}),\end{align*} where $C=C(M,\|c\|_{C^{0,\alpha}_\e(\Omega)}, \alpha, \e^{-1}d')>0$ and $d'= \operatorname{dist}(\Omega',\partial\Omega)$.\end{theorem}

\begin{lemma}[Exponential decay lemma]\label{exponentialdecaylemma}

Let $\Sigma \subset M$ be a two-sided embedded hypersurface. Assume that Fermi coordinates around $\Sigma$ are defined up to height $\rho>0$. Let $-\rho< t_0 \leq T_0<\rho$ and  $L_\e=\e^2\Delta - c(x)$, with $\min_{\{t_0\leq t\leq T_0\}} c\geq  c_0>0$. 

There are positive constants $\sigma$ and $\e_0$ (depending only on $M$, $\Sigma$ and $c_0$), such that for all $\e\in(0,\e_0)$ and $d \in [0,\frac{T_0-t_0}{2}]$, if $v\in C^{2,\alpha}(\{t_0 \leq |t|\leq T_0\})$ satisfies $L_\e v \geq -a$, with $a\geq 0$, then \begin{equation}\label{nexp}
     \max_{\{t_0+d\leq |t|\leq T_0-d\}} v \leq 2 \| v \|_{L^\infty(M)} e^{-\sigma d/\e}+\frac{a}{c_0}.
\end{equation}
\end{lemma}

\begin{proof}
First, we deal with the case $a=0$. Let $h(t)=\|v\|_{L^\infty(M)}(e^{-\sigma (t-t_0)/\e}+e^{\sigma(t-T_0)/\e})$, where $\sigma>0$ will be chosen later. Then $\e^2 \Delta h\leq  [\sigma^2 + \e (\|H_t\|_{L^\infty(\{|t|\leq \rho\}}+1)]h$. Assume that $v-h$ attains a positive maximum at $(x',t')$. Since $v-h<0$ in $\{t=t_0\}\cup \{t=T_0\}$, this must be an interior point, i.e. $(x',t')\in\{0<t<\rho\}$. Therefore, at $(x',t')$, $0 \geq \e^2 \Delta (v-h)\geq [c_0-\sigma^2-\e (\|H_t\|_{L^\infty(\{|t|\leq \rho\}}+1)](v-h)$. Since we are assuming $(v-h)(x',t')>0$, we obtain a contradiction as long as  $0<\e<\e_0=(c_0/2)(\|H_t\|_{L^\infty(\{|t|\leq \rho\}}+1)^{-1}$ and $\sigma=\frac{\sqrt{c_0}}{2}$. In conclusion, with these choices of $\e$ and $\sigma$ we have \begin{equation*}
    v(x,t)\leq \|v\|_{L^\infty(M)}(e^{-\sigma( t-t_0)/\e}+e^{\sigma(t-T_0)/\e}) \text{ on } \{t_0\leq |t|\leq T_0\}
\end{equation*} which implies \cref{nexp} when $a=0$.  
To obtain the case $a>0$, we set $\tilde v= v-\frac{a}{c_0}$. Then, $L_\e \tilde v\geq 0$ and \cref{nexp} follows from applying the case $a=0$ to $\tilde v$.

\end{proof}

We can use \cref{exponentialdecaylemma} together with the following lemma, to obtain estimates on solutions to $\cref{AC}$ far from their nodal set. This is the content of \cref{expocoro}.

\begin{lemma}
\label{ulower} There exists $\e_0 \in (0,\operatorname{inj}(M))$ and $\nu>0$, such that if $u$ is a solution of \cref{AC} for $\e\in(0,\e_0)$ that does not vanish on $B(p,\e)\subset M$, then $ |u(p)|\geq \nu$.
\end{lemma}

\begin{proof} Without loss of generality, we can assume $u>0$ on $B(p,\e)$. Let $u_{\e }$ be the only positive solution to \cref{AC} with zero Dirichlet boundary condition on $B(p,\e )$, which exists by \cref{existence} (after rescaling if necessary). Note that, by \cref{below}, we have $u(p)\geq u_{\e }(p)$. In addition, $\e$-rescalings of $u_{\e }(p)$ around the point $p$ converge to $\tilde u$, the unique positive solution to \cref{DAC} on $B(0,1)\subset \R^n$. Therefore, $\liminf_{\e\to 0} u(p)\geq \tilde u(0)>0$. This proves the statement.
\end{proof}

\begin{corollary}\label{expocoro} Fix $k\geq 0$ and let $\Sigma$ and $\rho$ be as in \cref{exponentialdecaylemma}. There is an integer $m_k$, such that for every $0<r+\e<R<\rho$ and $u$ satisfying  \cref{AC} with nodal set $\{u=0\} \subset \{|t|<r\}$ we have  $$\| \sgn(u)-u \|_{C^{k,\alpha}_\e(M\setminus \{|t|<R\})} =O(\e^{-m_k}e^{-(R-r)/\e}).$$
\end{corollary}

\begin{proof} On $M\setminus \{u=0\}$ we can rewrite \cref{AC} as $\e^2\Delta v - c v=0,$ where $v=1-u$ and $c=|u|^2+|u|$. From  \cref{ulower}, it follows that $|u|\geq \nu$ on $M\setminus \{|t|<r+\e\}$. Therefore, $c$ is uniformly bounded from below in this region.
By  \cref{exponentialdecaylemma}, and since $0\leq 1-u =v
\leq 1$, there exists $C>0$ such that
\begin{align*}\sup\limits_{\partial \{|t|<R\}}|v| &\leq C e^{-(R-r)/\e},
\end{align*} for all $R\in(r+\e,\rho)$. By the maximum principle, $\|v \|_{L^\infty(M\setminus \{|t|<R\})}\leq \sup\limits_{\partial \{|t|<R\}}|v|$. It follows from  \cref{NashSchauder} that $\|v \|_{C_\e^{2,\alpha}(M\setminus \{|t|<R\})}=O(\e^{-n/2} e^{-(R-r)/\e})$. The estimate for $C^{k,\alpha}_\e$ follows from iteratively applying Schauder estimates to derivatives of $v$.
\end{proof}


\section{Characterizations of entire one dimensional solutions}\label{entireclass}

\begin{definition}An entire solution $u$ of \cref{AC} in $\R^n$ is said to be \textit{one dimensional} if there are $p,v\in \R^n$, with $|v|\leq 1$, such that $u(x)=\psi(\e^{-1}(x-p)\cdot v)$. \end{definition}

A one dimensional solution has parallel planar level sets, with its profile in the orthogonal direction to these planes being a translation of $\psi(t/\e)$. Characterizing such solutions is the first step in order to obtain curvature estimates for the level set of general solutions. In the late 70s, De Giorgi conjectured that entire monotone bounded solutions in $\R^n$ of \cref{AC} should be one dimensional, at least for $n\leq 8$. This is now known to be true for $n=2,3$ (see \cite{GhoussoubGui} and \cite{AmbrosioCabre}, respectively) and false for $n\geq 9$ (see \cite{DelPino}). In dimensions $4\leq n\leq 8$, it is known  under the additional hypothesis of Savin's Theorem (see \cite{Savin,Wang}) which is an analogue of Bernstein's Theorem. At the present time, entire stable solutions of the Allen-Cahn are known to be one-dimensional only when $n=3$ and they have finite multiplicity at infinity (\cite{AmbrosioCabre}).

\

In this section, we summarize two characterizations of one dimensional solutions. First, they are the only entire solutions with multiplicity one at infinity (see \cite{Wang}, and  \cref{wangentire}). Second, they are the only entire solutions having its nodal set enclosed in between two parallel planes (see \cite{Farina} and  \cref{gibbons}).

\

For solutions with multiplicity one at infinity, we have the following theorem (see \cref{wanglocal}, for a local version).
\begin{theorem}[K. Wang, \cite{Wang}]\label{wangentire} There is $\tau_0\in\R$, such that if $u$ is an entire solution to \cref{AC}, with $\e=1$, in $\R^{n+1}$, then $$\lim_{R\to\infty }R^{-n}\int_{B_{R}} \frac{|\nabla u|^2}{2}+W(u)\leq (1+\tau_0)\omega_n\sigma_0,$$ implies that $u$ is one dimensional.
\end{theorem}

For solutions with nodal set contained between two parallel planes we use the following version of Gibbons conjecture due to Farina (\cite{Farina}).

\begin{theorem}\label{gibbons}
Let $u: \R^n \to \R$ be a solution to $\Delta u- W'(u)=0$. Assume the nodal set $\{u=0\}$ is contained in a slab $\{ x \in\R^n : |\langle x, v\rangle| \leq K \}$, for some $K\geq 0$ and some direction $v\in S^{n-1}$. Then, $u$ is one dimensional, i.e $u(x)=\pm\psi((x-dv)\cdot v)$, for some $|d|\leq K$. 
\end{theorem}

Actually, Farina proved \cref{gibbons} under the additional assumption that the sets $\Omega_{+}=\{ u >0 \}$ and $\Omega_{-}=\{ u <0 \}$ are both unbounded in the direction orthogonal to the hyperplanes (see Theorem 2.1 of \cite{Farina}). In \cref{omegaunbounded} below, we show that this hypothesis is not necessary. Note that this represents a breaking point in the analogy between solutions of \cref{AC} and minimal hypersurfaces, as there are examples of rotationally symmetric three dimensional catenoids contained in a slab of $\R^4$ with finite height (see Section 2 of \cite{Tam}).

\begin{lemma}\label{omegaunbounded}
Let $u: \R^N \to \R$ be a solution to $\Delta u- W'(u)=0$. Assume the nodal set $\{u=0\}$ is contained in a slab $\{ x \in\R^n : |\langle x, v\rangle| \leq K \}$, for some $K\geq 0$ and some direction $v\in S^{N-1}$. Then, $\Omega_{\pm}$ are both unbounded on the direction of $v$.
\end{lemma}

\begin{proof}[Sketch of the proof]

Without loss of generality assume $v=e_n.$ Denote the coordinate functions of $\R^n=\R^{n-1}\times \R$ by $(x,y)$. Let $(x_n,y_n)\in \{u=0\}$ be a sequence such that $y_n\to y_0 = \sup_{\{u=0\}}y$. Define $u_n(x,y)=u(x-x_n,y)$ then $u_n\to \tilde u$ smoothly on compacts sets. Moreover, the limit $\tilde u$ is an entire solution with nodal set enclosed by two horizontal planes and such that the nodal set attains its maximum height at $(0,y_0)\in\{\tilde u = 0\}$. At this step, we use a sliding argument together with the maximum principle to show that $\tilde u$ is itself a one dimensional solution, concluding that the sets $\{\tilde u>0\}$ and $\{\tilde u<0\}$ are both unbounded. Finally, note that outside of the region enclosed by the horizontal planes the signs of $\tilde u$, $u_n$ and $u$ are the same. 
\end{proof}




   
\section{Proof of \cref{zerolevelset}}\label{strictlystable}

\cref{zerolevelset} is derived in several steps. First, once appropriate barriers have been constructed, a sliding argument shows that the distance from the nodal set $\{u=0\}$ to the limit interface $\Gamma$ is of order $O(\e)$. This implies that the nodal set of $\e$-blow-ups of $u$ near $\Gamma$, is bounded between two horizontal planes in $\mathbb{R}^n$. From the characterization of  \cref{gibbons}, they converge to an entire one-dimensional solution. Finally, we obtain the desired energy estimates combining the analysis near $\Gamma$ with \cref{exponentialdecaylemma}.

The following result, whose proof we delay to the next section, summarizes the existence of the barriers necessary for the sliding argument. We emphasize that it is crucial that $\Gamma$ is strictly stable. In this case there is a tubular neighborhood of $\Gamma$ which is foliated by strictly mean-convex constant mean curvature leaves (except the central leaf which is $\Gamma$ itself). It is because of the strict mean-convexity, that each leaf gives rise to a barrier with the right sign. 

\begin{lemma}\label{barriers}
There exist $C>0$, $\e_0>0$ and $r_0>0$, depending only on $\Gamma$ and $U$, such that for each $\e\in(0,\e_0]$, there is a continuous family of functions $v_{H} \in C^{2,\alpha}(M)$, depending on the parameter $ H \in(C\e, H_\max]$ such that:

\begin{enumerate}
    \item $\e^2 \Delta v_H - W'(v_H)>0$,
    \item $\{v_H=0\}\subset \{H-C\e<|t|<H+C\e\}$, and
    \item $v_H<\operatorname{sgn}-\frac{1}{C}\e^2$ in $\{t=r_0\}$.
\end{enumerate}
\end{lemma}

Assuming \cref{barriers} we are now ready to prove \cref{zerolevelset}.

\begin{proof}[\textbf{Proof of  \cref{zerolevelset}.}]
Let $u=u_\e$ be the sequence of solutions to \cref{AC} whose nodal set $\{u=0\}$ converges towards $\Gamma$ in the Hausdorff distance as $\e\to 0$. Therefore, for $\e>0$ sufficiently small, we have $\{u=0\}\subset\{|t|\leq r_0/3\}$. Working with $-u$ instead of $u$ and making $\e$ smaller if necessary, we can also assume that $u>1/2 \text{ in } \{r_0/3\leq t\leq 3 r_0\}$ and $|u|>1/2$ in $\{-3r_0<t<-r_0/3\}$. Note that we are making no assumptions on the sign of $u$ for negative values of $t$. By \cref{exponentialdecaylemma}, these bounds are improved to $u-1=O(\e^\N)$ in  $\{r_0/2<t<2r_0\}$ and $u-\sgn(u)=O(\e^\N)$ in  $\{-2r_0<t<-r_0/2\}$.

Let $v_H$ be the family constructed in   \cref{barriers}. The estimates on $u$ from the last paragraph, together with the hypothesis on $v_H$ (i.e. (3) \cref{barriers}), imply that $u>v_H$ in $\{|t|=r_0\}$, for all $H\in(C\e,H_\max]$. Moreover, for $\e$ sufficiently small, $\{u=0\}$ is contained in the negative region of $v_{H_\max}$. In fact, note that $H_\max$ is fixed, so by (2) \cref{barriers}, the nodal set of $v_H$ is far from $\Gamma$, while $\{u=0\}$ is converging to $\Gamma$ in the Hausdorff distance. For all $H \in (C\e,H_\max]$ the function $v_H$ is a subsolution to \cref{AC}, so, from \cref{below}, we conclude $u>v_{H_\max}$. By the maximum principle, this inequality must hold for all $H \in (C\e,H_\max]$. Therefore, $\{u=0\}\subset \{t<2C\e\}$. Repeating the argument for negative values of $t$ and working with $-u$ if necessary, we conclude that $\{u=0\}\subset \{|t|<2C\e \}.$

Blowing up $u$ in Fermi coordinates over $\Gamma$, produces entire solutions of $\Delta u-W'(u)=0$ in $\R^n$, with nodal set between two parallel planes.  \cref{gibbons} implies that these have to be one dimensional solutions with horizontal level sets. It follows that for any $R>0$, $$\lim_{\e\to 0} \int_{\{|t|<\e R\}} \e\frac{|\nabla u|^2}{2}+\frac{W(u)}{\e} \leq \sigma_0 |\Gamma|.$$ 

Moreover, from the local converge in $\{|t|<r\}$ together with the exponential decay \cref{exponentialdecaylemma}, it follows that for $R>0$ large enough we have $|\sgn(t)-u(x,t)| \leq C e^{-\sigma t/\e}$ on $\{\e R <|t|<\e^{\ds}\}$ and $u=o(\e^\N)$ in $M\setminus \{|t|<\e^{\ds}\}$. This implies $$\int_{\{\e R<|t|<\e^\ds\}} \e\frac{|\nabla u|^2}{2}+\frac{W(u)}{\e} = O( |\Gamma|\times e^{-2R})$$ and $$\int_{M\setminus \{|t|>\e^\ds\}} \e\frac{|\nabla u|^2}{2}+\frac{W(u)}{\e} = o(\e^\N).$$

Combining all the estimates, and since $R>0$ can be choosen arbitrarily large, we conclude that $$\lim_{\e\to 0} \int_{M} \e\frac{|\nabla u|^2}{2}+\frac{W(u)}{\e} \leq \sigma_0 |\Gamma|.$$ 

 \end{proof}


\section{Proof of \cref{barriers}}\label{nbarriers}

In this section we use a cut-off of $\psi(t/\e)$ which is $0$ for large values of $|t|$. To be precise, fix values $0<r<R$ to be chosen later, and denote by $\chi:\mathbb{R}\to\mathbb{R}$ a smooth function with $\chi\equiv 1$ for $|t|\leq r$, $\chi\equiv 0$ for $|t|\geq R$ and such that $\chi$ is strictly monotone on $r\leq |t|\leq R$. 

\begin{definition}\label{ndefomega} Let $\omega:\R\to\R$ be given by \begin{equation*}
    \begin{split}
        \omega(t)&=\psi(t/\e)\chi(t). 
    \end{split}
\end{equation*}
\end{definition} 

Since $\Gamma$ is strictly stable it has a local neihgborhood which is foliated by mean-convex constant mean curvature graphs. We now construct barriers adapted to these hypersurfaces.  Let $\Sigma_{H_0} \subset M$ denote 
a two-sided, connected, closed, embedded hypersurface with constant mean curvature $H_0>0$. Keeping this in mind, we work in Fermi coordinates $(x,t)$ with respect to the normal vector field $\nu$ over $\Sigma_{H_0}$, which is opposite to $\vec H(\Sigma_{H_0})$. Therefore, we have $$H_0=-\langle \vec H(\Sigma_{H_0}), \nu\rangle>0.$$ 

We also assume $R$ in \cref{ndefomega} has been chosen smaller than the maximum height of the Fermi coordinates. This allows us to think of $\omega$ as a function on $M$, by means of $\omega(x,t)=\omega(t)$. \\ 

\noindent\textbf{Notation.} \textit{Given $|c|< R$, we denote by $\{t=c\}$ the points of $M$ which have height exactly $c$ in Fermi coordinates around $\Sigma_{H_0}$. Similarly, $\{a<t<b\}$ denotes the union of all the $\{t=c\}$ with $c\in(a,b)$.}\\
 
Let $Q(v)=\e^2 \Delta v-W'(v)$. A short computation shows its linearization around $\omega$ is given by
\begin{equation}\label{nlinear}
    \begin{split}
        Q(\omega+\phi)=Q(\omega)+L_\omega(\phi)+\phi^2(\phi-3\omega)
    \end{split}
\end{equation}
where $L_\omega=\e^2\Delta-W''(\omega)$. Our goal is to find $\sigma>0$ and $\phi:M\to \mathbb{R}$ satisfying $$Q(\omega+\phi)=\e H_0 \sigma +O(\e^2)$$ in a region of the form $\{|t|\leq r\}$ and such that $\omega+\phi$ is smaller than $\operatorname{sgn}$ on the boundary $\{|t|= r\}$.

First, we rewrite \cref{nlinear} as
\begin{equation}\label{neq1}
    \begin{split}
        Q(\omega+\phi)-\e H_0\sigma=\e H_0(\psi'(t/\e)\chi-\sigma)+L_\omega(\phi)+E_1+E_2+E_3+E_4
    \end{split}
\end{equation}
where
\begin{equation*}
    \begin{split}
    E_1=&\phi^2(\phi-3\omega)\\
E_2=&-\e^2 t \omega'\int_{0}^1 \langle \partial_s \vec H(x,ts),\partial(x,ts) \rangle ds\\ 
E_3=&2\e \psi'(t/\e)\chi'+[\e^2H_0\psi(t/\e)\chi'+\e^2 \psi(t/\e)\chi'']\\
E_4=&W'(\psi(t/\e))\chi - W'(\omega).
    \end{split}
\end{equation*}

Estimating these error terms will simplify the notation later.
\begin{remark}
 Clearly, $|E_1|=O(|\phi|^2)$. Next, we have $|E_2|=O(\e^2)$. This bound follows since the function $(t/\e)\psi'(t/\e)$ has  bounded ${C^{k,\alpha}_\e(M)}$-norm. For $E_3$, the first term is $o(\e^\N)$ since $\psi'(t/\e)$ decays exponentially fast on $\operatorname{supp} \chi'=\{r\leq |t|\leq R\}$. The remaining terms are $O(\e^2) $. Therefore, $|E_3|=O(\e^2).$ Lastly, $E_4=0$ in $\{|t|\leq r\}$, since in this region we have $\chi\equiv 1$ and $\omega\equiv\psi(t/\e)$.
 
\end{remark}
Using the estimates from the remark and defining $$f=\e H_0(\psi'(t/\e)\chi-\sigma),$$ on the region on $\{|t|<r\}$ we can rewrite \cref{neq1} as
\begin{equation}\label{ndes1}
    \begin{split}
        Q(\omega+\phi)=\e H_0\sigma+f+L_\omega(\phi)+O(|\phi|^2+\e^2). 
    \end{split}
\end{equation}

\noindent\textbf{Choosing $\sigma$ and $\phi$.} First, we pick $\sigma$ so that $f\perp \psi'(t/\e)\chi$, i.e.\begin{equation}
    \begin{split}\label{nsigma}
\sigma&=\frac{\int (\psi'(t/\e)\chi)^2}{\int \psi'(t/\e)\chi}.
    \end{split}
\end{equation} 

Next, let $v_1 \in C^{2,\alpha}_\e(M)$ be the unique solution to $L_\infty v_1=-f$. We define $\phi_1=v_1^\perp$, i.e.
\begin{equation}\label{nphi1}
    \begin{split}
\phi_1&=v_1-  \bigg(\frac{\int v_1 \psi'(s/\e)\chi ds}{\int (\psi'(s/\e)\chi)^2 ds}\bigg)  \psi'(t/\e)\chi.
    \end{split}
\end{equation}
Denote $g=f+L_\omega \phi_1$ and let $0<\tilde r<r$. We define \begin{equation}\label{nphi2}
    \begin{split}
\phi_2=v_2 \chi
    \end{split}
\end{equation} where $v_2\in C^{2,\alpha}_\e(\Sigma_{H_0} \times \mathbb{R})$ is the unique solution to $L_0 v_2=-(g\tilde\chi)^\perp_0=-g\tilde\chi + \bigg(\frac{\int g\tilde\chi \psi'(s/\e) ds}{\int \psi'(s/\e)^2 ds}\bigg)\psi'(t/\e)$ with $\tilde\chi\equiv 1$ for $|t|<\tilde r$ and $\tilde \chi\equiv 0$ for $|t|>r$.  Using Fermi coordinates, we ought to interpret $\phi_2$ as belonging to $C^{2,\alpha}_\e(M)$.

Finally, from these explicit formulas, it follows that $$\phi=\phi_1+\phi_2$$ depends continuously on $H_0$. 

\

  \noindent\textbf{Estimates.} From \cref{injectivitycylinder}, \cref{coercive}, \cref{nsigma}, \cref{nphi1}, \cref{nphi2} and since $\|f\|_{C^{0,\alpha}_\e(M)}=O(\e)$, we immediately have
\begin{equation}\label{1stest}
    \begin{split}
     \|\phi_1\|_{C^{2,\alpha}_\e(M)}+\|\phi_2\|_{C^{2,\alpha}_\e(M)}+\|v_1\|_{C^{2,\alpha}_\e(M)}+\|v_2\|_{C^{2,\alpha}_\e(\Sigma_{H_0}\times \mathbb{R})}=O(\e).
     \end{split}
\end{equation} 

Our next goal is to estimate the term $f+L_\omega\phi$ in \cref{ndes1}. First, we rewrite it as
\begin{equation*}
    \begin{split}
       f+L_\omega \phi=(L_\omega -L_0)\phi_2 + (g\tilde\chi+L_0\phi_2)+g(1-\tilde\chi).
    \end{split}
\end{equation*}

Estimates for $g$ follow from the expression:
\begin{equation}\label{ng}
    \begin{split}
    g=& L_\omega \phi_1 +f\\
    =& L_\omega \phi_1 - L_\infty v_1\\
    =&\e^2\Delta (\phi_1-v_1)-W''(\omega)\phi_1+2v_1\\
    =&v_1(2-W''(\omega))-[\e^2\Delta-W''(\omega)] v_1^T\\
    =&3v_1(1-\omega^2)+[\e^2\Delta_t+\e^2\partial_t^2+\e^2 H_t\partial_t-W''(\omega)] v_1^T\\
    =&3v_1(1-\omega^2)+ \frac{\int(\e^2\Delta_t v_1)\psi'\chi}{\int (\psi'\chi)^2}\psi'\chi +  \frac{\int v_1\psi'\chi}{\int (\psi'\chi)^2} (\e^2\partial_t^2+\e^2 H_t\partial_t-W''(\omega))\psi'\chi.\\
            =&3v_1(1-\omega^2)+ \frac{\int(\e^2\Delta_t v_1)\psi'\chi}{\int (\psi'\chi)^2}\psi'\chi+ \\&+  \frac{\int v_1\psi'\chi}{\int (\psi'\chi)^2} [\e H_t\psi''\chi+\e^2H_t\psi'\chi'+ \chi(W''(\psi)-W''(\omega))\psi'+2\e\psi''\chi'+\e^2\psi'\chi''].\\
    \end{split}
\end{equation}

\begin{claim}\label{nclaim1}
There exists $C>0$, independent of $\e$, such that
$$\|g\tilde \chi\|_{C^{0,\alpha}_\e(\{|t|>s\})}\leq C\e e^{-|s|/\e}$$ for all $|s|\leq R$.
\end{claim}

\begin{proof} The bounds follow from $\|v_1\|_{C^{2,\alpha}_\e(M)}=O(\e)$ and the fact that $1-|\psi(t)|$ and all of its derivatives decay exponentially, i.e. there exists $C>0$, such that \begin{equation}\label{npsiest}
    \begin{split}\|1-|\psi(t/\e)|\|_{C^{0,\alpha}_\e(\{|t|>s\})}+\|(\partial_t^k\psi)(t/\e)|\|_{C^{0,\alpha}_\e(\{|t|>s\})}\leq Ce^{-|s|/\e} \end{split}
\end{equation} for all $k=1,2,3$. These are applied to the expansion of $g\tilde\chi$ given by \cref{ng}. Note that in $\operatorname{supp}\tilde\chi$, we have $\rho\equiv \chi\equiv 1$. So $\omega=\psi(t/\e)$ in this region. The first term is $3v_1(1-\psi(t/\e)^2)\tilde \chi$. The $C^{0,\alpha}_\e$-norm of the integrals is bounded by $\|v_1\|_{C^{2,\alpha}_\e(M)}=O(\e)$, using item (3) from \cref{nnorms}. 
\end{proof}

\begin{claim}\label{nclaimphi}
There exist $C>0$, independent of $\e$, such that $$\| \phi_2\|_{C^{0,\alpha}_\e(\{|t|\geq s\})}\leq C\e e^{-|s|/\e}$$ for all $|s|\leq R$.
\end{claim}

\begin{proof}
There exists a constant $s_0>0$ such that $W''(\psi(t/\e))\geq 1$ for all $t\geq s_0\e$. In addition, from \cref{nclaim1} we have $$\|L_0 v_2\|_{C^{0,\alpha}_\e(\{|t|>s\})}=\|g\tilde \chi^\perp\|_{C^{0,\alpha}_\e(\{|t|>s\})}\leq C\e e^{-|s|/\e}.$$  Therefore, we can apply \cref{exponentialdecaylemma} to both $v_2$ and $-v_2$, with $L_\e=L_0$, $c_0=1$, $a=C\e e^{-|s|/\e}$ and $\Omega=\{|t|>\e s_0\}$, to get $$\|v_2\|_{L^{\infty}(\{|t|>s\})}\leq C \e e^{-|s|/\e}$$ for all $s>\e s_0$, where we redefine $C$ if necessary. By Schauder estimates, the same bound holds for the $C^{2,\alpha}_\e$-norm of $v_2$. This implies the claim since $\phi_2=v_2\chi$.
\end{proof}

\begin{claim}
$\|(L_\omega-L_0)\phi_2\|_{L^\infty(M)}=O(\e^2)$.
\end{claim}

\begin{proof}
Rewrite the expression as $(L_\omega - L_0)\phi_2 = \e^2(\Delta_t - \Delta_0) \phi_2 + \e^2H_t\phi_2'+(W''(\psi(t/\e))-W''(\omega))\phi_2.$
We estimate each term separately. From the proof of item (4) in \cref{nnorms} and \cref{nclaimphi}
\begin{equation*}
    \begin{split}
   |\e^2(\Delta_t-\Delta_0)\phi_2|=&O(|t\partial_{ij}\phi_2|+ |t\partial_{i}\phi_2|)
   =O(\e\times |t/\e|\times (\e e^{-|t|/\e}))=O(\e^2).
    \end{split}
\end{equation*}
Next $|\e^2 H_t\phi'_2|=O(\e \|\phi_2\|_{C^1_\e(M)})=O(\e^2).$ Finally, $\operatorname{supp}[W''(\psi(t/\e))-W''(\omega)]\subset \{|t|>r\}$. From \cref{nclaimphi} in this region $|\phi_2|=O(\e e^{-r/\e})=o(\e^\N).$ Therefore $|(W''(\psi(t/\e))-W''(\omega))\phi_2|=o(\e^\N).$ 
\end{proof}

\begin{claim}
$\|g\tilde\chi+L_0\phi_2\|_{L^\infty(M)}=O(\e^2)$
\end{claim}

\begin{proof} Consider the expansion
\begin{equation*} 
    \begin{split}
    g\tilde\chi+L_0\phi_2&=g\tilde\chi+L_0(v_2\chi)\\
    &=g\tilde\chi + (L_0v_2)\chi +\e^2(2v_2'\chi'+v_2\chi'')\\
    &=g\tilde\chi - (g\tilde\chi)_0^\perp \chi+\e^2(2v_2'\chi'+v_2\chi'')\\
    &=g\tilde\chi(1-\chi)+ (g\tilde\chi)_0^T\chi+\e^2(2v_2'\chi'+v_2\chi'').
    \end{split}
\end{equation*}
The first term vanishes since $\tilde \chi(1-\chi)\equiv 0$. For the last term, we have $|\e^2(2v_2'\chi'+v_2\chi'')|=O(\e\|\phi_2\|_{C^{1}_\e(M)}+\e^2\|\phi_2\|_{L^\infty(M)})=O(\e^2).$ For the second term, we first need to estimate
\begin{equation*}
    \begin{split}
  \int \tilde \chi \psi'(t/\e) L_\omega \phi_1=& \int  \tilde \chi \psi'(t/\e) (\e^2\Delta_t+\e^2\partial_t+\e^2 H_t\partial_t-W''(\omega))\phi_1\\
  =& \int  \tilde \chi \psi'(t/\e) \e^2(\Delta_t-\Delta_0)\phi_1+\e^2 \Delta_0 \int \phi_1\chi\psi'(t/\e) \\&+ \int (\e^2\Delta_0\phi_1)(\tilde \chi-\chi)\psi'(t/\e)+\int \e^2 H_t(\partial_t\phi_1)\tilde\chi\psi'(t/\e)\\
  &+\int \psi'(t/\e)\tilde\chi(\e^2\partial_t^2-W''(\omega))\phi_1.
    \end{split}
\end{equation*}
By the proof of item (4) from \cref{nnorms}, the first term is bounded by $\e\|\phi_1\|_{C^{2,\alpha}_\e(M)}\int|t/\e\psi'(t/\e)|=O(\e^3)$. The second term is zero by the definition of $\phi_1$. The third term is $o(\e^\N)$ since $\e^2\Delta_0\phi_1$ is bounded, $\psi'(t/\e)$ decays exponentially and $\operatorname{supp}(\tilde\chi-\chi)$ is far from the origin. The fourth term is bounded by $O(\e\|\phi_1\|_{C^{1,\alpha}_\e(M)}\int \psi'(t/\e))=O(\e^3)$. Since $\omega=\psi(t/\e)$ in $\operatorname{supp}\tilde \chi$, the last term is $\int \psi'(t/\e)\tilde\chi \ell_0\phi_1$. Integrating by parts, using $\ell_0 \psi(t/\e)=0$ and arguing as we did for the third term, we see this term is also $o(\e^\N)$. In other words 

$$\int \tilde \chi \psi'(t/\e) L_\omega \phi_1=O(\e^3).$$

Finally, we have
\begin{equation*}
    \begin{split}
    \int g\tilde \chi\psi'(t/\e)&=\int (L_\omega \phi_1+f)\tilde\chi \psi'(t/\e)\\
    &=\int (L_\omega \phi_1)\tilde\chi \psi'(t/\e)+ \int f\chi \psi'(t/\e)+\int f(\tilde\chi-\chi)\psi'(t/\e)\\
    &=\int (L_\omega \phi_1)\tilde \chi \psi'(t/\e)+o(\e^\N)\\
    &=O(\e^3)
    \end{split}
\end{equation*}
where the second term on the second line is zero by the definition of $\sigma$ and the third term is $o(\e^{\N})$ arguing similarly as the previous estimate. This implies $\|(g\tilde\chi)^T_0\|_{L^\infty(\Sigma\times \mathbb{R})}=O(\e^2)$ from which the claim follows. 
\end{proof}

\begin{claim}
  $g(1-\tilde \chi)=O(\e^\N)$ in $\{|t|<r\}$.
\end{claim}

\begin{proof}
Remember $\operatorname{supp}(1-\tilde \chi)=\{t \leq -\tilde r\}\cup \{\tilde r \leq t \}$. So from \cref{ng}, and as in \cref{nclaim1}, we can write $
g(1-\tilde \chi)=3v_1(1-\omega^2)(1-\tilde \chi)+o(\e^\N).$ By construction, $1-|\omega|=1-|\psi(t/\e)|=o(\e^\N)$ on $\{\tilde r\leq |t| \leq r\}$. Since $v_1$ is bounded, this implies $g(1-\tilde \chi)=o(\e^\N)$ in $\{|t|\leq r\}$. 
\end{proof}

Summarizing, we have shown that $f+L_\omega \phi=O(\e^2)$ on $\{|t|<r\}.$ Substituting into \cref{ndes1}, and since the convergence of $\lim_{\e\to 0}\sigma=\sigma_0=\frac{\int \psi'(t)}{\int \psi'(t)^2}>0$ is exponentially fast, we see that, for our choices of $\sigma$ and $\phi$, $$Q(\omega+\phi)=\e H_0 \sigma_0 + O(\e^2) \text{ on } \{|t|<r\}.$$  This already gives us ranges for $H_0$ and $\e$, where $Q(\omega+\phi)>0$. However, in order to control the boundary conditions for our barrier argument, we need slightly more room in our inequality.

\begin{claim}
There exist $C>0$, $\lambda>0$ and $\e_0>0$, such that \begin{equation}\label{newineq}
    Q(\omega+\phi)-\e^2\lambda>0 \text{ on } \{|t|<r\}
\end{equation} for all $\e\in(0, \e_0]$ and $H_0\in[C\e,H_\max]$.
\end{claim}

\begin{proof}
There exists $C>0$, such that the right hand side of $Q(\omega+\phi)-\e\frac{H_0\sigma}{2}=\e\frac{H_0\sigma}{2}+O(\e^2)$, is strictly positive as long as $C\e\leq  H_0\leq H_\max$ and $\e$ is small enough. If we denote $\lambda=\frac{\sigma C}{2}$, then $\e^2\lambda \leq \e\frac{H_0\sigma}{2}$ and the inequality follows.
\end{proof}

Finally, we are ready to estimate the boundary conditions. To simplify the notation, let $v=\omega+\phi$ and $c=\e^2\lambda$. Note that $Q(v)-c = \e^2\Delta v -(W'(v)+c)>0$. The idea is similar to that of \cref{expocoro}. On one hand, because of \cref{newineq}, $v=\omega+\phi$ accumulates near $v_-$ and $v_+$, where $v_-<v_0<v_+$ are the roots of the polynomial $W'(v)+c=(v-v_-)(v-v_0)(v-v_+)$. On the other hand, these roots are distanced from $-1$ and $1$ by a fixed amount. In fact, a simple computation shows that if $c>0$ is small enough there exists $k>0$, independent of $c$, such that: 
\begin{equation}\label{nroots}
    \begin{cases}
v_-<-1-kc\\
0<v_0<kc\\
0<v_+<1-kc.
\end{cases}
\end{equation}

\begin{claim}
$v<\operatorname{sgn}-k\e^2\lambda+o(\e^\N)$ on  $\{r/3<|t|<r/2\}$, for all $H_0\in [C\e,H_\max]$.
\end{claim}

\begin{proof}
By construction, $v=\omega+\phi=\psi(t/\e)+O(\e)$ in $\{|t|<r\}$. In particular, $v=1+O(\e)$ in $\{r/4<t<r\}$. Together with \cref{nroots} this implies $v-v_0>1-kc+O(\e)=1+O(\e)$ and $v-v_->[1+O(\e)]+[1+kc]=2+O(\e)$, since $kc=O(\e^2)$. Therefore, $$\e\Delta (v-v_+)-c_+(v-v_+)\geq 0 \text{ on } \{r/4<t<r\}$$ where $c_+=\frac{W'(v)+c}{v-v_+}=(v-v_-)(v-v_0)=2+O(\e)>1$, if $\e$ is small enough. \cref{exponentialdecaylemma}  then implies $ v-v_+<o(\e^\N)$ in $\{r/3<t<r/2\}$. 

Similarly, $v=-1+O(\e)$ in $\{-r<t<-r/4\}$. Together with \cref{nroots} this implies $v-v_0<-1+O(\e)$ and $v-v_+<-1+O(\e)$. Therefore, $$\e\Delta (v-v_-)-c_-(v-v_-)\geq 0 \text{ on } \{r/4<t<r\}$$ where $c_-=\frac{W'(v)+c}{v-v_-}=(v-v_+)(v-v_0)=1+O(\e)>1/2$, if $\e$ is small enough. \cref{exponentialdecaylemma} then implies $v-v_-<o(\e^\N)$ in $\{-r/2<t<r/3\}$.
\end{proof} 

Combining both inequalities with \cref{nroots} finishes the proof of \cref{barriers}.




\section{Curvature Estimates for multiplicity one solutions}\label{curvature}

In this section, we present the following curvature estimates 

\begin{lemma}\label{curvatureestimates}
Let $\Gamma \subset M$ be a non-degenerate minimal hypersurface, which is also the limit interface for a sequence of solutions to \cref{AC} with multiplicity one. Then, for $\e=\e(\Gamma,M)$ small enough, the nodal set of the solutions is a normal graph $\Gamma(f)$, with $$\|f\|_{C^{2}(\Gamma)}+\e^{\alpha}\|f\|_{C^{2,\alpha}(\Gamma)}=O(\e)$$ for some $\alpha\in(0,1)$.
\end{lemma} As in \cite{Mantoulidis,ChodoshMantoulidis}  these estimates are derived from the work of Wang and Wang-Wei, combined with a standard point-picking and blow-up argument.

In the computations below, we rely on the following two theorems which were proven for the case of $\R^n$ in \cite{Wang} and \cite{WangWei}, respectively. The proof for general ambient manifolds with bounded curvature tensor, follows the same strategy with minor modifications.

\begin{theorem}[K. Wang, see \cite{Wang}]\label{wanglocal} Let $M$ be a closed Riemannian manifold and $R_0\in (0,\operatorname{inj}(M))$. There are positive numbers $\e_0, \tau_0,\alpha_0\in (0,1)$, $r_0\in(0,R_0)$ and $K_0$, such that the following holds. If $u$ is a solution of \cref{AC} on $B_{R_0}(p)\subset M$, with $\e<\e_0$, $p\in M$, $u(p)=0$, and $$R_0^{-n}\int_{B_{R_0}}\e \frac{|\nabla u|^2}{2}+\frac{W(u)}{\e}\leq (1+\tau_0)\omega_n \sigma_0,$$ then there exists a hyperplane $\pi\subset T_p M$, such that  $\{u=0\}\cap \exp_p(B_{r_0,\pi}(0)\times[-r_0,r_0])$ is a normal graph (with respect to Fermi coordinates on $\pi$) over the ball  $B_{r_0,\pi}(0)\subset \pi$. Moreover, the $C^{1,\alpha_0}(B_{r_0,\pi}(0))$ norm of this graph is bounded by $K_0$.
\end{theorem}

\begin{theorem}[Wang-Wei, see Section 15 of \cite{WangWei}]\label{wangwei} Let $M$ be a closed Riemannian manifold. Let $u_i: B_R(p_i)\to \R$ be a sequence of solutions to \cref{AC} with $\e=\e_i\to 0$. Assume that \begin{enumerate}
\item [i)] $\{u_i=0\}$ is, in exponential coordinates, a normal graph over a hyperplane $\pi_i\subset T_{p_i}B_R$, with Lipschitz constant uniformly bounded on $i$ and converging to a smooth hypersurface as $i\to 0$.
\item  [ii)] For any $q_i\in \{u_i=0\}$ the blow-ups $\tilde u_i(x)=u_i \circ \exp_{q_i}(\e x)$ converge to a one dimensional solution in $\R^{n+1}$, and
\item  [iii)] The second fundamental form of $\{u_i=0\}$ is  bounded uniformly on $i$.
\end{enumerate}
Then, on a smaller ball $B_r\subset B_R$,  the mean curvature of $\{u_i=0\}$ satisfies $$|H|_{C^{0}(\pi_i)}+\e^\alpha |H|_{C^{0,\alpha}(\pi_i)}=O(\e).$$ Moreover, the $C^{2,\alpha}$ norm of $\{u_i=0\}$ a graph over $\pi$ is bounded.
\end{theorem}

\

\begin{proof}[\textbf{Proof of \cref{curvatureestimates}}]

\setcounter{claim}{0}

\begin{claim}For $\e$ small enough, the nodal set $\{u=0\}$ is an embedded hypersurface and the second fundamental form of $u$ near $\{u=0\}$ is of order $o(\e^{-1})$.\end{claim}

By the monotonicity formula, $\e$-rescalings of multiplicity one solutions centered at the nodal set, have multiplicity one at infinity and therefore are one dimensional by  \cref{wangentire}. This implies $|\nabla u|\neq 0$ on the nodal set, if $\e$ is small enough. The estimate on the second fundamental form follows from the smoothness of the convergence of the rescalings to the 1-D solution, which has planar level sets.

\begin{claim} For $\e$ small enough, the nodal set is a normal graph $\Gamma(f)=\{u=0\}$ for some $f\in C^{\infty}(\Gamma)$. Moreover, the Lipschitz norm of $f$ is uniform on $\e$.\end{claim}

For each $p\in \{u=0\}$,  \cref{wanglocal} implies that $\{u=0\}$ is a graph over some hyperplane $\pi \in T_p M$, with uniform $C^{1,\alpha}$ bounds on a ball of fixed radius. If there is a sequence of $p$ and $\pi$, converging to a hyperplane orthogonal to $\Gamma$, then the $C^{1,\alpha}$ bound would imply there is concentration of energy far from $\Gamma$, which we are assuming  does not happen. It follows that, for $\e$ small, each component of $\{u=0\}$ is a normal graph over $\Gamma$ with Lipschitz norm uniform with respect to $\e$. By the mu

\begin{claim}The second fundamental form of $\Gamma(f)$ is $O(1)$.\end{claim}

Let $p\in \{u=0\}$ be the point where the norm of the second fundamental form of the nodal set attains its maximum, which we denote by $\lambda$. From Claim 1, we have $\e\lambda\to 0$. To argue by contradiction, assume $\limsup \lambda =\infty$. Then, for any $R>0$ and $\e$ small enough, the rescalings $v(x):=u \circ \exp_p(x/\lambda)$ are solutions to $(\e\lambda)^2 \Delta v - W'(v)=0$ in $B_{R/\e}(0)\subset T_pM$ with respect to the metric $g_\lambda=\lambda^{-2}\exp_p^*(g)$. Since $( T_pM,g_\lambda)$ converges to $\R^n$ in compact sets around the origin, it follows from the monotonicity formula and the multiplicity one assumption, that the limit varifold  is a hyperplane. Restricting to $B_1(0)$, we obtain a list of solutions, with nodal set converging to a hyperplane in the Hausdorff distance. Moreover, by Claim 2, the nodal set is also a uniformly bounded Lipschitz graph with second fundamental form bounded from above by $1$. We can now apply \cref{wangwei}, which implies that the $C^{2,\alpha}$ norm of this graph is universally bounded and therefore it must converge in $C^{2,\alpha'}$ to the hyperplane, for $0<\alpha'<\alpha$ and after passing to a subsequence if necessary. This contradicts that the norm of the second fundamental form at the origin is exactly 1. Therefore, it must be that $\lambda=O(1)$.

\begin{claim}
$\|f\|_{C^{2,\alpha}(\Gamma)}+\e^{\alpha}\|f\|_{C^{2,\alpha}(\Gamma)}=O(\e).$ 
\end{claim}

Finally, we can apply \cref{wangwei} to our original sequence of solutions and conclude that its mean curvature satisfies $|H|_{C^{0}(\Gamma)}+\e^\alpha |H|_{C^{0,\alpha}(\Gamma)}=O(\e)$ and $|f|_{C^{2,\alpha}(\Gamma)}$ is bounded, where $\pi$ is a plane in $T_p\Gamma$. Note that since $\{u=0\}\to \Gamma$ in Hausdorff, by compactness we must have $|f|_{C^{2,\alpha'}(\Gamma)}\to 0$, for $0<\alpha'<\alpha$. Finally, we use that $\Gamma$ is non-degenerate. Since the Jacobi operator of $\Gamma$ is invertible, by the Inverse Function Theorem, the map $H:\mathcal{U}\subset C^{2,\alpha'}(\Gamma) \to C^{0,\alpha'}(\Gamma)$ has a continuous inverse on some small open set $\mathcal{U}$ around the constant zero function. Therefore, when $|f|_{C^{2,\alpha'}(\Gamma)}$ is small enough, we have $$|f|_{C^{2}(\Gamma)}+\e^\alpha |f|_{C^{2,\alpha'}(\Gamma)}=O(|H|_{C^{0}(\Gamma)}+\e^\alpha |H|_{C^{0,\alpha}(\Gamma)})=O(\e).$$

\end{proof}

\section{Proof of \cref{B}}\label{uniqueness}

Let $u$ be a solution of \cref{AC} converging to a non-degenerate minimal hypersurface $\Gamma$ with multiplicity one. From \cref{curvature} we know that for $\e$ sufficiently small, $\Gamma(f)=\{u=0\}$ with $\|f \|_{C^2(\Gamma)}+\e^\alpha \|f\|_{C^{2,\alpha}(\Gamma)}=O(\e).$

\begin{definition}Given $\xi \in C(\Gamma)$ with $|\xi|=o(1)$, we denote $\omega_\xi(x,t)=\omega(x,t-\xi(x)).$ Similarly, $\omega'_\xi(x,t)=\omega'(t-\xi(x))$, $\omega''_\xi(x,t)=\omega''(t-\xi(x))$ and so on. 
\end{definition}

\begin{remark}\label{localconv}Together with  \cref{wangentire}, the estimates above imply that the rescaling of both functions converges to the same one dimensional solution. More precisely, for any fixed $R>0$ $$\| u-\omega_f\|_{C_\e^2(\{|t|<\e R\})}=o(1).$$   \end{remark} 

We begin this section by looking for a perturbation of $f$ of the form $\xi=f+h$, and such that the error $\phi=u-\omega_{\xi}$ is orthogonal to the approximate kernel $\omega_\xi'$ in the following sense:
\begin{definition}
A smooth function $\phi: M\to \R$ is said to be \textit{orthogonal to the approximate kernel} $\omega'_\xi$ if for all $x\in \Gamma$, \begin{align}\label{perp}\int_\R \phi(x,t)\omega'_\xi(t)dt=0.\end{align}
\end{definition}

\begin{remark}\label{lineofarg}
As in \cite{WangWei} and \cite{ChodoshMantoulidis},  \cref{perp} allows for the following procedure. First, $L^2$ estimates for $\phi$ are obtained from \cref{linearizedapprox} and \cref{perp}. Then, these are improved to estimates of the $C^{2,\alpha}_\e$-norm, using  \cref{NashSchauder}, \cref{Schauder} and the $C^{0,\alpha}_\e$ norm of $\e^2 \Delta \phi - W''(\omega_\xi)\phi$.
\end{remark}

In this section, we carry out an argument following the lines described in \cref{lineofarg}.

\begin{proposition}\label{hapriori}
There exists $\xi \in C^\infty(\Gamma)$ such that the error $\phi=u-\omega_{\xi}$ is orthogonal to the approximate kernel $\omega'_{\xi}$. In addition, \begin{align*}
\|\phi \|_{C_\e^{2,\alpha}(M)}=o(1) \ \ \text{ and }  \ \
\|\nabla_0^k \xi \|_{C_\e^{0,\alpha}(\Gamma)}= O(\e+ \e^{1-k}\|  \phi\|_{C^{k,\alpha}_\e(M)}),\end{align*} for $k=0,1,2$ and $\alpha$ as in  \cref{curvatureestimates}.
\end{proposition}

\begin{proof}[Proof of  \cref{hapriori}] 

Let $U=\{h\in C(\Gamma): |h|<\tau/2\}$ and $F$ be the map $F:U \to C(\Gamma)$, given by \begin{align*}F(h)(x):&=\e \int_\R [u(x,t)-\omega_{f+h}(x,t)]\omega'_{f+h}(x,t)dt.\end{align*}
From \cref{localconv} and  \cref{cutoffestimates}-(6) we obtain

\setcounter{claim}{0}

 \begin{claim} $F(0)=o(\e).$ \end{claim}

Similarly, we can estimate $|DF(h)|$ from below when $\|h\|_{C(\Gamma)}$ is small. Denote by $B(f,r)\subset C(\Gamma)$, the ball of radius $r>0$ centered at $f\in C(\Gamma)$, with respect to the supremum norm. Let $r=o(\e)$ and $h \in B(0, r)$.

\begin{claim}
For $\e$ small enough, $DF(h)(v)=c v, \forall v\in C(\Gamma),$ where $c=c(h)\geq \sigma_2/2$. In particular, $B(F(0),\frac{c}{2}r)\subset F(B(0,r))$. \end{claim}

Indeed, from \cref{cutoffestimates} (6) and (7) we get \begin{align*}DF(h)(v) &=\frac{d}{ds}F(h+sv)|_{s=0}\\
&=v\cdot  \e \bigg[  \int_\R (\omega'_{f+h})^2   -\int_\R [u-\omega_{f+h}]\omega''_{f+h}     \bigg]\\
&=v \cdot [\sigma_1 + o(\e^\N) +o (1)],
\end{align*}
which implies the claim for $\e$ small enough.

\begin{claim}
There exists $h \in C^\infty(\Gamma)$, satisfying $\|h\|_{C(\Gamma)}=o(\e)$ and $F(h)\equiv 0$.
\end{claim} 

To see this, choose $r=o(\e)$ such that $F(0)=o(r)$, e.g. $r=\sqrt{\e F(0)}$. Since $F(0)=o(\e)$, the last claim implies $0\in B(F(0),\frac{c}{2}r)$, for $\e$ sufficiently small. Therefore, $0\in F(B(0,r))$.

\begin{claim}
Let $\xi=f+h$, where $h$ is as in the previous claim. Then, $|\nabla_0^k \xi |=o(\e^{1-k}),$ for $k=1,2,3$. In particular, $\|\phi\|_{C^{k}_\e(M)}=o(1),$ for $k=1,2,3$.
\end{claim}

Now that we have guaranteed the existence of $h$ with $\|h\|_{C(\Gamma)}=o(\e)$, its smoothness follows from applying the Implicit Function Theorem and the nondegeneracy of $DF(h)$ to the function $\tilde F :\Gamma \times (-\e,\e) \to \R$ given by $\tilde F(x,h):=F(h)(x)$. Then, the estimate $|\nabla_0^k \xi |=o(\e^{1-k}),$ for $k=1,2,3$, follows by recursively differentiating $F(h)(x)=0$ with respect to $\nabla_0^k$ and estimating the norm of the result, each time using \cref{cutoffestimates}.

Finally, we have to argue for $\|\phi\|_{C^{k}_\e(M)}=o(1)$. From  \cref{expocoro}, for every $\rho>0$ we can choose $R=O(\e)$ such that $\|\phi\|_{C^{k}_\e(M\setminus \{|t|<R\})}<\rho$, for $\e$ small enough. In exponential coordinates on points of $\{|t|<R\}$, the function $\omega_\xi$ rescales as $\psi(t-\tilde \xi(x))$, where $\tilde \xi(x)=\xi(\e x)/\e$. From the previous estimate it follows that $\nabla_0^k \tilde \xi =o(1)$, for $k=1,2,3$. This implies $\psi(t-\tilde \xi(x))$ converges to the canonical solution in all $C^k$ norms. The same is true for $u$ from \cref{localconv} and putting both estimates together we conclude that for any $\rho>0$, there is an $\e_0>0$ such that $\|\phi\|_{C^{k}_\e(M)}\leq \rho$ for $k=1,2,3$ and $\e\in(0,\e_0)$.

\begin{claim}
$\|\nabla_0^k h \|_{C^{0,\alpha}_\e(\Gamma)}=O(\e^{1-k})\| \phi\|_{C^{2,\alpha}_\e(M)},$ for  $k=1,2$ and $\alpha\in[0,1)$.
\end{claim}

The estimates $\|\nabla_0^k h \|_{C^{0,\alpha}_\e(\Gamma)}=O(\e^{1-k})\| \phi\|_{C^{2,\alpha}_\e(M)}$ are obtained recursively by differentiating $0=u(x,f(x))=\omega_{f+h}(-h(x))+\phi(x,f(x))$. 

\end{proof}


Directly from the definition of $\phi$ and \cref{AC}, using the notation for Fermi coordinates from  \cref{Fermicoordinates} we can write the following equation for $\phi$ and $\xi$

\begin{align}\label{phiequation}
\begin{split}
\e^2 \Delta_0 \phi +\ell_0(\phi)+ E_1 &=\e^2\Delta_g\phi - W''(\omega_\xi)\phi\\
& = \phi^2(2\omega_\xi+u) + \e^2 J[\xi]\omega'_\xi + E_2 \\
\end{split}
\end{align}
where \begin{itemize}
\item $\ell_0(\phi)=\e^2\phi''-W''(\omega_\xi)\phi$,\\
\item $J[\xi]=\Delta_0 \xi -H'_0\xi$ is the Jacobi operator of $\Gamma$,\\
\item $E_1=-\e^2 H_t \phi'+\e^2 (\Delta_t-\Delta_0)\phi$,\\
\item $E_2=\e^2\omega_\xi'''|\nabla_t \xi|^2+\e^2 [ (\Delta_0-\Delta_t)\xi] \omega'_\xi+\e^2 H'_0 (t-\xi) \omega'_\xi + \e^2  \mathcal{R} t^2 \omega'_\xi$ and\\ \item $\mathcal{R}=\mathcal{R}(x,t)=\int_{0}^{1} H''(x,t\cdot s) (s-1)ds.$\end{itemize}

\begin{remark}
Although this equation involves both $\phi$ and $\xi$, it follows from  \cref{hapriori} that the right hand side of all the estimates can be presented in terms of norms of $\phi$. This is what we do in the rest of this section. 
\end{remark}

First, we compute the estimates for $J[\xi].$

\begin{proposition}\label{Jschauder}
$\|\e J[\xi]\|_{C^{0,\alpha}_\e(M)}=O(\e^3+\e\|\phi \|_{C^{2,\alpha}_\e(M)}+\|\phi \|^2_{C^{2,\alpha}_\e(M)})$
\end{proposition}

\begin{proof}

We project onto $\Gamma$ by integrating against $\omega'_\xi$ along every vertical direction. 

\

For the first term, notice that from \cref{Ortho0} we have the expression, 
\begin{align*}
\bigg\|\int_\R \e^2 (\Delta_0 \phi)  \omega'_\xi \bigg\|_{C^{0,\alpha}_\e(\Gamma)}
&=\bigg\|-2\e^2 \nabla_0\xi  \int_\R (\nabla_0\phi)   \omega''_\xi-\e^2|\nabla_0 \xi|^2\int_\R  \phi  \omega'''_\xi + \e^2 \Delta_0 \xi  \int_\R \phi  \omega''_\xi \bigg\|_{C^{0,\alpha}_\e(\Gamma)} \\
&=O(\|\nabla_0\xi\|_{C^{0,\alpha}_\e(M)}+ \e\|\nabla^2_0\xi\|_{C^{0,\alpha}_\e(M)}) \|\phi\|_{C^{2,\alpha}_\e(M)}.
\end{align*}

For the second term, by the formula in the introduction, we have
\begin{align*}
\bigg\| \int_\R \ell (\phi )\omega'_\xi \bigg\|_{C^{0,\alpha}_\e(\Gamma)} &=\bigg\| \int_\R [\ell_0(\phi) +o(\e^\N)]\omega'_\xi \bigg\|_{C^{0,\alpha}_\e(\Gamma)}= o(\e^\N).
\end{align*}

Next,
\begin{align*}
\begin{split}
\bigg\|\int_\R \e^2 H_t\phi'\omega'_\xi \bigg\|_{C^{0,\alpha}_\e(\Gamma)}&=O(\e^2)\|H_t \|_{C^{1}_\e(M)}\|\phi' \|_{C^{1}_\e(M)}=O(\e\|\phi\|_{C^2_\e(M)})
\end{split}
\end{align*}

Now, from the Fermi Coordinates section, we know that 
\begin{align*}
 \bigg\|\int_\R [\e^2(\Delta_t-\Delta_0) \phi(\cdot,t)] \omega'_\xi  \bigg\|_{C^{0,\alpha}_\e(\Gamma)}&=O(1) \int_\R \| \e^2(\Delta_{t+\xi}-\Delta_0) \phi(\cdot,t+\xi)\|_{C^{0,\alpha}_\e(\Gamma)} |\omega'| \\
&=O(1) \|\phi \|_{C^{2,\alpha}_\e(M)} \int_\R (|t|+\e ) |\omega'|\\
&=O(\e) \|\phi \|_{C^{2,\alpha}_\e(M)}.
 \end{align*}

Since $\|u\|_{C^{2,\alpha}_\e(M)}+\|\omega_\xi\|_{C^{2,\alpha}_\e(M)}=O(1),$ we have
\begin{align*}
\bigg\| \int_\R \phi^2(2\omega_\xi+u)\omega'_\xi \bigg\|_{C^{0,\alpha}_\e(\Gamma)} &=O( \|\phi \|^2_{C^{1}_\e} )
\end{align*}

Similarly, 
\begin{align*}
\bigg\|\int_\R \e^2 J[\xi] (\omega'_\xi)^2 \bigg\|_{C^{0,\alpha}_\e(\Gamma)}&=( \e \sigma_{2} + o(\e^\N)) \|J[\xi]\|_{C^{0,\alpha}_\e(\Gamma)}
\end{align*}

It remains to estimate the error $E_2$,

First, since $\omega'''$ is an odd function
\begin{align*}
\bigg\|\int_\R \e^2 |\nabla_t \xi|^2 \omega'''_\xi\bigg\|_{C^{0,\alpha}_\e(\Gamma)}&=
\bigg\|\int_\R\e^2( |\nabla_t \xi|^2 - |\nabla_0 \xi|^2 ) \omega'''_\xi\bigg\|_{C^{0,\alpha}_\e(\Gamma)}
\\
&=
\int_\R \e^2 \| |\nabla_{t+\xi} \xi|^2 - |\nabla_0 \xi|^2 \|_{C^{0,\alpha}_\e(\Gamma)}
 |\omega'''|\\
 &=\e^2\|\nabla_0\xi\|^2_{C^{0,\alpha}_\e(\Gamma)}\int_\R (|t|+\e)|\omega'''|\\
 &=O(\e)\|\nabla_0\xi\|^2_{C^{0,\alpha}_\e(\Gamma)}
\end{align*}

Next, we have 
\begin{align*}
\bigg\|\int_\R (\e^2 (\Delta_{t}-\Delta_0) \xi) (\omega'_\xi)^2 \bigg\|_{C^{0,\alpha}_\e(\Gamma)}&=O(1) \int_\R \|\e^2(\Delta_{t+\xi}-\Delta_0) \xi\|_{C^{0,\alpha}_\e(\Gamma)} (\omega_\xi')^2   \\
&=  O(\|\nabla_0 \xi \|_{C^{0,\alpha}_\e(\Gamma)}+\|\nabla_0^2 \xi \|_{C^{0,\alpha}_\e(\Gamma)}) \int_\R (|t|+\e)  \times (\e \omega_\xi')^2   \\
&=O(\e^2) (\|\nabla_0 \xi \|_{C^{0,\alpha}_\e(\Gamma)}+\|\nabla_0^2 \xi \|_{C^{0,\alpha}_\e(\Gamma)}).
\end{align*}

Since $t(\omega')^2$ is an odd function,
\begin{align*}
\bigg\|\int_\R \e^2 H'_0 (t-\xi)  (\omega'_\xi)^2  \bigg\|_{C^{0,\alpha}_\e(\Gamma)}&=\bigg\|\e^2 H'_0 \int_\R  t ( \omega')^2  \bigg\|_{C^{0,\alpha}_\e(\Gamma)}=0.
\end{align*}

Finally, since $\|t+\xi \|_{C^{0,\alpha}_\e(M)}=O(|t|+\|\xi\|_{C^{0,\alpha}_\e(M)})=O(|t|+\e)$, we have
\begin{align*}
\bigg\|\int_\R \e^2 \mathcal{R} t^2 ( \omega'_\xi)^2 \bigg\|_{C^{0,\alpha}_\e(\Gamma)}&= O(\e^2) \|\mathcal{R} \|_{C^{0,\alpha}_\e(M)} \int_\R (|t|^2+\e^2) (\omega')^2 \\
&= O(\e^2) \int_\R (|t/\e|^2+1) (\e \omega')^2 \\
&=O(\e^3).
\end{align*}

Combining all the estimates, we obtain
\begin{align}\label{J-estimate}\begin{split}
\|\e J[\xi]\|_{C^{0,\alpha}_\e(\Gamma)}&= O(\|\nabla_0\xi\|_{C^{0,\alpha}_\e(M)})(
\e^2+\e\|\nabla_0\xi\|_{C^{0,\alpha}_\e(M)}+\|\phi\|_{C^{2,\alpha}_\e(M)}) \\
&+ O(\|\nabla^2_0\xi\|_{C^{0,\alpha}_\e(M)}) (\e^2+\e \|\phi\|_{C^{2,\alpha}_\e(M)})\\ 
&+ O(\e^3+\e \|\phi \|_{C^{2,\alpha}_\e(M)}+ \|\phi \|^2_{C^{2,\alpha}_\e(M)})\\
\end{split}\end{align}

Finally, substituting $\|\nabla_0^k \xi \|_{C^{0,\alpha}_\e(\Gamma)}=O(\e+\e^{1-k}\|\phi\|_{C^{2,\alpha}_\e(M)})$ into \cref{J-estimate} we get
\begin{align*}
\|\e J[\xi]\|_{C^{0,\alpha}_\e(\Gamma)}&= O(\e^3+\e\|\phi \|_{C^{2,\alpha}_\e(M)}+\|\phi \|^2_{C^{2,\alpha}_\e(M)}).
\end{align*}
\end{proof}

Now we compute estimates for the value of the approximated linearized operator at $\phi$.

\begin{proposition}\label{linearizedschauder}
$\|\e^2 \Delta_g \phi -W''(\omega_\xi)\phi  \|_{C^{0,\alpha}_\e(M)}=O(\e^2+\| \phi\|^2_{C^{2,\alpha}_\e(M)}).$
\end{proposition}

\begin{proof}
Remember 
\begin{align}
\begin{split}
\e^2 \Delta_g \phi - W''(\omega)\phi & = \phi^2(2\omega_\xi+u) + \e J[\xi]\dot \omega_\xi + E_2 \\
\end{split}
\end{align}
where \begin{itemize}\item $J[\xi]=\Delta_0 \xi -H'_0\xi$ is the Jacobi operator of $\Gamma$\\ \item $E_2=W'(\omega_\xi)|\nabla_t \xi|^2 +\e^2 [ (\Delta_t-\Delta_0)\xi+H'_0 (t-\xi) + t^2 \mathcal{R} ] \omega'_\xi$ and\\ \item $\mathcal{R}=\mathcal{R}(x,t)=\int_{0}^{1} H''(x,ts) (s-1)ds.$\end{itemize}

\

We estimate each term of $\|E_2 \|_{C^{0,\alpha}_\e(M)}$  separately
\begin{align*}\| W'(\omega_\xi) |\nabla_t \xi|^2\|_{C^{0,\alpha}_\e(M)}&= O(\|\nabla_0 \xi\|^2_{C^{0,\alpha}_\e(M)})\\
&=O(\e+\|\phi\|_{C^{1}_\e(M)})^2\\
&=O(\e^2+\|\phi\|^2_{C^{2,\alpha}_\e(M)}).
\end{align*}

\begin{align*}\| \e^2 (\Delta_t-\Delta_0)\xi \omega'_\xi \|_{C^{0,\alpha}_\e(M)}&=O(\e^2 \|\nabla^2_0\xi \|_{C^{0,\alpha}_\e(M)}+\e^2 \|\nabla_0\xi \|_{C^{0,\alpha}_\e(M)})\| t \omega'_\xi(x,t) \|_{C^{0,\alpha}_\e(M)}\\
&=O(\e^2 \|\nabla^2_0\xi \|_{C^{0,\alpha}_\e(M)}+\e^2 \|\nabla_0\xi \|_{C^{0,\alpha}_\e(M)})\|( t/\e) \e \omega'_\xi(x,t) \|_{C^{0,\alpha}_\e(M)}\\
&=O(\e^2 \|\nabla^2_0\xi \|_{C^{0,\alpha}_\e(M)}+\e^2 \|\nabla_0\xi \|_{C^{0,\alpha}_\e(M)})\\
&=O(\e^2)(\e+\e^{-1}\|\phi\|_{C^{2,\alpha}_\e(M)})\\
&=O(\e^3+\e\|\phi\|_{C^{2,\alpha}_\e(M)}).
\end{align*}

\begin{align*}
\|\e^2 H'_0 (t-\xi) \omega'_\xi\|_{C^{0,\alpha}_\e(M)}&=O(\e^2) \|H'_0\|_{C^{0,\alpha}_\e(M)}\|(t-\xi) \omega'_\xi\|_{C^{0,\alpha}_\e(M)}=O(\e^2).
\end{align*}

\begin{align*}
\| \e^2 t^2 \mathcal{R} \omega'_\xi\|_{C^{0,\alpha}_\e(M)}&=O(\e^2)\|\mathcal{R} \|_{C^{0,\alpha}_\e(M)}\| t^2 \omega'_\xi\|_{C^{0,\alpha}_\e(M)}=O(\e^3).
\end{align*}

Collecting all these estimates with the ones for $J[\xi]$ from the last section, we conclude 
\begin{align*}
\| \e^2 \Delta_g \phi - W''(\omega)\phi \|_{C^{0,\alpha}_\e(M)}& = \|\phi^2\|_{C^{0,\alpha}_\e(M)} + \|\e J[\xi]\|_{C^{0,\alpha}_\e(M)} + \|E_1\|_{C^{0,\alpha}_\e(M)}\\
&=O(\e^2+\|\phi\|_{C^{2,\alpha}_\e(M)}^2)
\end{align*}

\end{proof}


The following three lemmas estimate the $L^\infty$ norm of $\phi$.

\begin{lemma}[$L^\infty$-norm estimate far from $\Gamma$]
There exist positive constants $\sigma$, $R_0$ such that, for all $R\geq R_0$, $$\|\phi\|_{L^\infty(M\setminus \{|t|<\e R\})}=O(e^{-\sigma R}\|\phi\|_{C^{2,\alpha}_\e(M)} + \e^2+\|\phi\|^2_{C^{2,\alpha}_\e(M)}).$$ 
\end{lemma}

\begin{proof}
Since $|\xi|=O(\e)$, there are positive constants  $R_0$ and $\e_0$, such that $W''(\omega_\xi)>1$ on $M\setminus \{|t|<\e R/2\}$ for $\e\in (0,\e_0)$. In particular, we can apply  \cref{exponentialdecaylemma} with $\|\e^2\Delta_g \phi-W''(\omega_\xi)\phi\|_{C^{2,\alpha}_\e(M)}=O(\e^2+\|\phi\|^2_{C^{2,\alpha}_\e(M)}),$ $\rho=\tau-\e R$, $\Omega=M\setminus \{|t|<\e R/2\}$, $c=W''(\omega_\xi)$, $c_0=1$ and $a=O(\e^2+\|\phi\|^2_{C^{2,\alpha}_\e(M)})$. In this way $$|\phi(p)|=O(\|\phi\|_{L^\infty(\partial \{|t|<\e R\}}\times \max\{e^{-\sigma \dist(p,\partial \{|t|<\e R/2\})/\e}, e^{-\sigma(\tau/\e- R/2)} \} + \e^2+\|\phi\|^2_{C^{2,\alpha}_\e(M)}).$$ Finally, when $p=(x,t)\in M \setminus \{|t|<\e R\}$ we obtain the desired estimate. \end{proof}

\begin{lemma}[$L^2$-norm estimate near $\Gamma$] For any fixed $R>0$, we have
 $$\e^{-n/2} \sup_{p\in M_{\e R}}\|\phi \|_{L^2(B(p,\e L))}=O\bigg(\e^2 + \|\phi\|_{C^{2,\alpha}_\e(M)}^2 + e^{-\sigma R} \|\phi\|_{C^{2,\alpha}_\e(M)} \bigg).$$
\end{lemma}

\begin{proof}
We start by computing the equation satisfied by the $L^2$-norm along vertical directions of $\Gamma$, i.e. $V(x)=\int_{J}|\phi(x,t)|^2 dt=\|\phi(x,\cdot)\|^2_{L^2(J)}$, where $J=[-2R\e,2R\e]$.

From the equation satisfied by $\phi$, derived above, we get
\begin{align*}
\begin{split}
\frac{\e^2}{2}\Delta_0 V &=   \int_{J} \phi (\e^2 \Delta_0 \phi) +  \int_{J}\e^2 | \nabla_0  \phi|^2\\
&\geq -\int_{J} \phi  \ell(\phi)+ \int_{J} \phi(E_2-E_1)+O(\phi^3)+ \e^2 J[\xi]\omega'_\xi\phi\\
&\geq -\int_{J} \phi  \ell(\phi)- \frac{\gamma}{4}\int_{J} \phi^2 + \frac{4}{\gamma}\int_J E_1^2+E_2^2+\e^4 (J[\xi]\omega'_\xi)^2.
\end{split}
\end{align*}

Define $\tilde\phi=\phi\rho$, where $\supp \rho \subset [-\tau,\tau]$ and $\rho\equiv1$ on $[-\tau/2,\tau/2]$. Notice we have $\int_\R \tilde\phi\omega'=\int_\R \phi\omega'=0$. Therefore, by the Lemma in the cutoff section we have
\begin{align*}
-\int_{J} \phi  \ell(\phi) &= -\int_{J} \tilde \phi  \ell(\tilde \phi) \\
&= -\int_{\R} \tilde \phi  \ell(\tilde \phi) + \int_{|t|\geq \e R} \tilde \phi  \ell(\tilde \phi)  \\
&\geq \frac{\gamma}{2} \int_\R \tilde\phi^2 + o(\e^\N)- \|\phi\|_{C^{2}_\e(M)} \int_{|t|\geq \e 2R} |\phi| \rho \\
&\geq \frac{\gamma}{2} V_0 + o(\e^\N)- \|\phi\|^2_{C^{2}_\e(M)} \int_{|t|\geq \e 2R} c e^{-\sigma t/\e} dt \\
&\geq \frac{\gamma}{2} V_0 + O(\e e^{-\sigma 2R}\|\phi\|^2_{C^{2}_\e(M)}) + o(\e^\N).\end{align*}

From the estimates we computed before we have 
\begin{align*}
|E_2|+|\e^2 J[\xi]\omega'_\xi|=|E_2|+|\e J[\xi]|=O(\e^2+\|\phi\|^2_{C^{2,\alpha}_\e(M)}),
\end{align*}
while, for $|E_1|$, since $H=0$, we have 
\begin{align*}
|E_1|=O(|\e^2 H_t \phi'|+|\e^2 (\Delta_t-\Delta_0) \phi|)=O(\|\phi\|_{C^{2,\alpha}_\e(M)})\times |t|,
\end{align*}
Since $|J|=2\e R$, for any fixed $R$ it follows $$\int_{-\e 2R}^{\e 2R} |E_1|^2+|E_2|^2+|\e^2 J[\xi]\omega'|^2=O(\e^5 + \e\|\phi\|_{C^{2,\alpha}_\e(M)}^4 + \e^3 \|\phi\|_{C^{2,\alpha}_\e(M)}^2).$$

Combining all the estimates the inequality for $V_0$ reads,
\begin{align*}
\begin{split}
\frac{\e^2}{2}\Delta_0 V - \frac{\gamma}{4}V_0 \geq  O(\e^5 + \e\|\phi\|_{C^{2,\alpha}_\e(M)}^4 + \e^3 \|\phi\|_{C^{2,\alpha}_\e(M)}^2 + \e e^{-\sigma 2 R}\|\phi\|^2_{C^{2}_\e(M)}).
\end{split}
\end{align*}

From the maximum principle we have 
 $$\|\phi(x,\cdot) \|^2_{L^2(-\e 2 R,\e 2 R)}=O(\e^5 + \e\|\phi\|_{C^{2,\alpha}_\e(M)}^4 + \e^3 \|\phi\|_{C^{2,\alpha}_\e(M)}^2 + \e e^{-\sigma 2 R}\|\phi\|^2_{C^{2}_\e(M)}).$$

Finally, notice that since $R>0$ is fixed, for $p=(x,t)$ with $t=O(\e)$ we have  
\begin{align*}
\|\phi \|^2_{L^2(B(p,\e R))}&=O(\e^{n-1})\times \|\phi(x,\cdot) \|^2_{L^2(-\e 2R,\e2 R)} \\
&=O(\e^{4+n} + \e^{n}\|\phi\|_{C^{2,\alpha}_\e(M)}^4 + \e^{2+n} \|\phi\|_{C^{2,\alpha}_\e(M)}^2 + \e^n e^{-\sigma 2 R}\|\phi\|^2_{C^{2}_\e(M)}),
\end{align*}
from which the estimate follows. 
\end{proof}

\begin{lemma}[$L^\infty$-estimate near $\Gamma$]
For any fixed $R>0$ and, we have $$\sup_{p\in M_{\e R}}\|\phi\|_{R^\infty(B(p,\e R/2))} = O\bigg(\e^2 + \|\phi\|_{C^{2,\alpha}_\e(M)}^2 + e^{-\sigma R} \|\phi\|_{C^{2,\alpha}_\e(M)} \bigg).$$
\end{lemma}

\begin{proof}
The proof is an immediate consequence of \cref{NashSchauder} and the estimates for $\|\phi\|_{R^2(B(p,\e R))}$ and $\|\e^2 \Delta_g \phi-W''(\omega_\xi)\phi \|_{L^\infty(M)}$. 
\end{proof}

\

Finally, we obtain the main technical result of this section.

\begin{corollary}\label{coroltech}
$\|\xi\|_{C^{2}(\Gamma)}+ \e^\alpha\|\xi\|_{C^{2,\alpha}(\Gamma)}+\|\phi\|_{C^{2,\alpha}_\e(M)} = O(\e^2 + \|\phi\|_{C^{2,\alpha}_\e(M)}^2).$
\end{corollary}

\begin{proof}
Combining both estimates we have the existence of $R_0$, such that for any fixed $R>R_0$, we have $$\|\phi\|_{L^\infty(M)} = O\bigg(\e^2 + \|\phi\|_{C^{2,\alpha}_\e(M)}^2 + e^{-\sigma R} \|\phi\|_{C^{2,\alpha}_\e(M)} \bigg).$$ Moreover, by the estimates from \cref{linearizedschauder} and  \cref{Schauder}, we are able to bound the $C^{2,\alpha}_\e(M)$-norm of $\phi$. We conclude by choosing $R$ big enough, which allow us to absorb the term $e^{-\sigma R} \|\phi\|_{C^{2,\alpha}_\e(M)}$ on the lefthand side. This proves the bound for $\phi$. From \cref{Jschauder} and the invertibility of $J$ it follows that $\|\xi\|_{C^{2}(\Gamma)}+ \e^\alpha\|\xi\|_{C^{2,\alpha}(\Gamma)} = O(\e^2 + \|\phi\|_{C^{2,\alpha}_\e(M)}^2).$ 
\end{proof}

The estimates obtained for the perturbation $\xi$ and the error $\phi$ are in Fermi coordinates with respect to $\Gamma$, and have the same order of homogeneity, with respect to $\e$, as in \cite{Pacard}. The only concern is that the error term is presented in a different format. More precisely, we have found $\phi$ so that $\int_\R \phi(x,t)\omega'(t-\xi(x))dt=0$. Denote by $D_\xi$ a diffeomorphism which in Fermi coordinates corresponds to $D_\xi(x,t)=(x,t+\xi(x))$ for $(x,t)\in \{|t|<\tau/2\}$ and that interpolates smoothly to the identity in $M\setminus \{|t|<\tau\}$ as $|t|\to \tau$. Let $v=\phi\circ D_\xi$. Then, $\int_\R v\omega '=0$. If we define $v^\sharp=[\chi_1 v]_0^\perp=\chi_1v-\frac{\int (\chi_1 v)\psi'(s/\e)ds}{\int \psi'(s/\e)^2ds}\psi'(t/\e)$ and $v^\flat=v-\chi_4 v^\sharp$ we get functions analogue to those in \cite{Pacard}, where $\chi_i$, for $i$ are defined in Section 3.5.1 of \cite{Pacard}. Pacard obtained his solutions by means of a contraction mapping argument, which implies the uniqueness of solutions with these asymptotics. This is explained in the paragraph after the proof of Lemma 3.9 from \cite{Pacard} (see also \cref{rem3} below). This concludes the proof of \cref{B}

\begin{remark}\label{rem3}
 In Lemma 3.9 from \cite{Pacard}, the bound for the auxiliary function $v^\flat$ is given in terms of a norm which is penalized near the hypersurface (see Remark 3.2 from \cite{Pacard} for the definition). More precisely, the lemma requires the bound $\|\chi_5 v^\flat \|_{C^{2,\alpha}_\e}(M)=O(\e^4)$. Using that $\int v\omega'=0$, and substituting the definitions of $v$, $[\cdot]_0^\perp$ and $\omega'$ into the formula of $v^\flat$, we get \begin{equation*}
     \begin{split}
        v^\flat=v(1-\chi_4\chi_1)+\chi_4\frac{\psi'(t/\e)}{\int \psi'(s/\e)^2 ds} \times \int v[(\chi_1-\chi)\psi'(s/\e)+\e(\psi(s/\e)-\operatorname{sgn})\chi']ds. 
     \end{split}
 \end{equation*} 
Note that $\chi_5(1-\chi_4\chi_1)=0$, so in order to estimate $\chi_5 v^\flat$ we only need to worry about the second term. On one hand, the derivatives of both $\chi_4\frac{\psi'(t/\e)}{\int \psi'(s/\e)^2 ds}$ and $v$, are bounded by polynomials on $\e^{-1}$. On the other hand, the functions $(\chi_1-\chi)\psi'(s/\e)$ and $\e(\psi(s/\e)-\operatorname{sgn})\chi'$ are both $o(\e^\N)$, this follows from \cref{onedimest} since the supports of $\chi_1-\chi$ and $\chi'$ are far from the hypersurface. Combining both facts, we see that the norms of $\chi_5 v^\flat$ decay exponentially on $\e$.
\end{remark}



\setcounter{tocdepth}{1}
\bibliographystyle{siam}
\bibliography{references}

\end{document}